\documentclass[leqno,11pt]{amsart}
\usepackage{amsmath,amstext,amssymb,amsopn,amsthm,mathrsfs}
\usepackage{psfrag}
\usepackage[dvips]{color}
\usepackage{multicol}
\usepackage{picinpar}
\usepackage{enumerate}
\usepackage[hyphens]{url}        
\usepackage[breaklinks,colorlinks=true,linkcolor=blue,citecolor=blue, urlcolor=blue]{hyperref}                                                 
\usepackage[utf8x]{inputenc}
\usepackage{txfonts}
\usepackage{graphicx}
\usepackage[T1]{fontenc}
\usepackage{longtable}
\theoremstyle{plain}
\newtheorem{teo}{Theorem}[section]
\newtheorem{defi}{Definition}[section]
\newtheorem{corollary}{Corollary}[section]
\newtheorem{lemma}{Lemma}[section]

\newtheorem{obs}{Observation}[section]
\numberwithin{equation}{section}

\def\re{\mathbb{R}}

\begin{document}

\title[Boundedness of the Gaussian Riesz potentials ] {Boundedness of the Gaussian Riesz potentials on Gaussian variable Lebesgue spaces}
\author{Eduard Navas}
\address{Departamento de Matem\'aticas, Universidad Nacional Experimental Francisco de Miranda, Punto Fijo, Venezuela.}
\email{enavas@correo.unefm.edu.ve}
   \author{Ebner Pineda}
\address{Departamento de Matem\'{a}tica,  Facultad de Ciencias Naturales y Matem\'aticas, ESPOL Guayaquil 09-01-5863, Ecuador.}
\email{epineda@espol.edu.ec}
\author{Wilfredo~O.~Urbina}
\address{Department of Mathematics, Actuarial Sciences  and Economics, Roosevelt University, Chicago, IL,
   60605, USA.}
\email{wurbinaromero@roosevelt.edu}


\begin{abstract}
In this paper we prove the boundedness of the Gaussian Riesz potentials $I_{\beta}$, for $\beta\geq 1$ on  $L^{p(\cdot)}(\gamma_d)$, the Gaussian variable Lebesgue spaces under a certain additional condition of regularity on $p(\cdot)$ following  \cite{DalSco}. Additionally, this result trivially gives us an alternative proof of the boundedness of Gaussian Riesz potentials $I_\beta$ on Gaussian Lebesgue spaces $L^p(\gamma_d)$.
\end{abstract}

\maketitle

\section{Introduction and Preliminaries}

In the classical case, the Riesz potential of order   $\beta>0$ is defined as negative fractional powers of the negative Laplacian $-\Delta = -\displaystyle\sum_{i=1}^d \frac{\partial^2}{\partial x_i^2}$,
\begin{equation}
 (-\Delta)^{-\beta/2},
\end{equation}
which means, using Fourier transform, that 
\begin{equation}
( (-\Delta)^{-\beta/2} f)\hat{}\;(\xi) = (2\pi |\xi| )^{-\beta} \hat{f}(\xi).
\end{equation}
for more details; see \cite{duo}, \cite{grafak}, \cite{st1}.\\

Analogously, the {\em  Gaussian fractional integrals} or {\em Gaussian Riesz potentials} can be also defined as negative fractional powers of the {\em Ornstein-Uhlenbeck operator }
\begin{equation}
(-L)= - \frac12 \Delta + \langle x, \nabla_x \rangle=- \sum_{i=1}^d \Big[\frac{1}{2} \frac{\partial^2}{\partial x_i^2} + x_i \frac{\partial }{\partial x_i}\Big].
\end{equation}
 However, since the Ornstein-Uhlenbeck operator has  eigenvalue $0,$  the negative powers are not defined on all of $L^2(\gamma_d),$ and therefore we need to be more careful with the definition. Let us consider 
 $$\Pi_{0}f=f-\displaystyle\int_{\mathbb{R}^{d}}f(y)\gamma_{d}(dy),$$
for $f\in L^{2}(\gamma_{d})$, the orthogonal projection on the orthogonal complement of  the eigenspace corresponding to the eigenvalue $0$.
\begin{defi}
The  Gaussian Fractional Integral or Gaussian Riesz potential of order $\beta>0$, $I_\beta$, is defined spectrally as,
\begin{equation}\label{i1}
I_\beta=(-L)^{-\beta/2}\Pi_{0},
\end{equation}
which means that for any multi-index $\nu, \; |\nu|>0$ its action on the Hermite polynomial $\vec{H}_\nu$ is given by
\begin{equation}\label{RieszPotAct}
I_\beta \vec{H}_\nu(x)=\frac 1{\left|
\nu \right|^{\beta/2}}\vec{H}_\nu(x),
\end{equation}
and for $\nu=0=(0,...,0), \, I_{\beta}(\vec{H}_{0})=0.$
\end{defi}
By linearity, using the fact that the Hermite polynomials are an algebraic basis of $\mathcal{P}(\re^d),$ $I_\beta$ can be defined for any polynomial function $f(x) = \sum_{\nu}  \widehat{f}_{\gamma_{d}}(\nu) \vec{H}_\nu(x),$ where $\widehat{f}_{\gamma_{d}}(\nu) = \frac{1}{\|\vec{H}_\nu\|_2} \int_{\mathbb{R}^d} f(t) \vec{H}_\nu (t) dt,$ as
\begin{equation}\label{RieszPotMult}
I_\beta f(x) = \sum_{\nu} \frac {\widehat{f}_{\gamma_{d}}(\nu)}{\left|\nu \right|^{\beta/2}}\vec{H}_\nu(x) = \sum_{k\geq 1} \frac 1{k^{\beta/2}} {\bf J}_k f(x),
\end{equation}
and similarly for $f \in L^2(\gamma_d),$ as the Hermite polynomials are an orthogonal basis of $L^2(\gamma_d).$\\

It can be proved that the Gaussian Riesz potential $I_\beta,$ $\beta>0,$ has the following integral representations, for  $f$ a polynomial function or $f \in C^2_b(\re^d),$
\begin{equation}\label{RieszOUIntRep}
I_\beta f(x)  =\frac 1{\Gamma(\beta/2)}\int_0^{\infty}
t^{\beta/2-1} T_t (I-{\bf J}_0)f(x)  \,dt,
\end{equation} 
with respect to the Ornstein-Uhlenbeck semigroup $\{T_t\}$, and
\begin{equation}\label{RieszPHIntRep}
I_\beta f(x)   = \frac 1{\Gamma(\beta)}\int_0^{\infty}
t^{\beta-1}P_t (I-{\bf J}_0)f(x)  \,dt,
\end{equation}
with respect to the Poisson-Hermite semigroup, $\{P_t\}$. \\

Therefore, from (\ref{RieszOUIntRep}) we have an explicit integral representation of $I_\beta$ as
\begin{equation}\label{RieszExplntRep}
I_{\beta}f(x)=\int\limits_{\mathbb{R}^{d}}N_{\beta/2}(x,y)f(y)dy\
\end{equation}
where the kernel $N_{\beta/2}$ is defined as
\begin{equation}\label{kernelRiesz}
N_{\beta/2}(x,y)=\frac{1}{\pi^{\frac{d}{2}}\Gamma(\frac{\beta}{2})}\int_{0}^{+\infty}
t^{\frac{\beta}{2}-1}\left(\frac{e^{-\frac{|y-e^{-t}x|^{2}}{1-e^{-2t}}}}{(1-e^{-2t})^{\frac{d}{2}}}-e^{-|y|^{2}}\right)dt
\end{equation}
For more details and background we refer to \cite{urbina2019}.\\

From (\ref{RieszPotMult}) it is clear that the Gaussian Riesz potentials $I_\beta$ are the simplest  Meyer's multipliers (see for instance Theorem 6.2 of \cite{urbina2019}), since in this case 
\begin{equation}\label{RieszPotMult2}
 m(k) = \frac{1}{k^\beta} = h( \frac{1}{k^\beta}),
\end{equation}
with $h(x) =x,$ the identity function and therefore their $L^p(\gamma_d)$-boundedness follows immediately.\\

In this paper we prove that Gaussian Riesz potentials $I_\beta$, for $\beta\geq 1$ are also bounded in $L^{p(\cdot)}(\gamma_d)$, the Gaussian variable Lebesgue spaces for certain exponent functions $p(\cdot)$ that will be determine later. For completeness, we will briefly review the notion of variable Lebesgue spaces.\\

Given $\mu$ a Borel measure, any $\mu_{-}$measurable function $p(\cdot):\mathbb{R}^{d}\rightarrow [1,\infty]$ is an $exponent$ $function$; the set of all the exponent functions will be denoted by  $\mathcal{P}(\mathbb{R}^{d},\mu)$. For $E\subset\mathbb{R}^{d}$ we set $$p_{-}(E)=\text{ess}\inf_{x\in E}p(x) \;\text{and}\; p_{+}(E)=\text{ess}\sup_{x\in E}p(x).$$
$\Omega_{\infty}=\lbrace x\in \Omega:p(x)=\infty\rbrace$.\\
We use the abbreviations $p_{+}=p_{+}(\mathbb{R}^{d})$ and $p_{-}=p_{-}(\mathbb{R}^{d})$.

\begin{defi}\label{deflogholder}
Let $E\subset \mathbb{R}^{d}$. We say that $\alpha(\cdot):E\rightarrow\mathbb{R}$ is locally log-H\"{o}lder continuous, and denote this by $\alpha(\cdot)\in LH_{0}(E)$, if there exists a constant $C_{1}>0$ such that
			\begin{eqnarray*}
				|\alpha(x)-\alpha(y)|&\leq&\frac{C_{1}}{log(e+\frac{1}{|x-y|})}
			\end{eqnarray*}
			for all $x,y\in E$. We say that $\alpha(\cdot)$ is log-H\"{o}lder continuous at infinity with base point at $x_{0}\in \mathbb{R}^{d}$, and denote this by $\alpha(\cdot)\in LH_{\infty}(E)$, if there exist  constants $\alpha_{\infty}\in\mathbb{R}$ and $C_{2}>0$ such that
			\begin{eqnarray*}
				|\alpha(x)-\alpha_{\infty}|&\leq&\frac{C_{2}}{log(e+|x-x_{0}|)}
			\end{eqnarray*}
			for all $x\in E$. We say that $\alpha(\cdot)$ is log-H\"{o}lder continuous, and denote this by $\alpha(\cdot)\in LH(E)$ if both conditions are satisfied.
			The maximum, $\max\{C_{1},C_{2}\}$ is called the log-H\"{o}lder constant of $\alpha(\cdot)$.
\end{defi}

\begin{defi}\label{defPdlog}
			We denote $p(\cdot)\in\mathcal{P}_{d}^{log}(\mathbb{R}^{d})$, if $\frac{1}{p(\cdot)}$ is log-H\"{o}lder continuous and  denote by $C_{log}(p)$ or $C_{log}$ the log-H\"{o}lder constant of $\frac{1}{p(\cdot)}$.
		\end{defi}
\begin{defi}
For a $\mu_{-}$measurable function $f:\mathbb{R}^{d}\rightarrow \overline{\mathbb{R}}$, we define the modular \begin{equation}
\rho_{p(\cdot),\mu}(f)=\displaystyle\int_{\mathbb{R}^{d}\setminus\Omega_{\infty}}|f(x)|^{p(x)}\mu(dx)+\|f\|_{L^{\infty}(\Omega_{\infty},\mu)},
\end{equation}
\end{defi}

\begin{defi} The variable exponent Lebesgue space on $\mathbb{R}^{d}$, $L^{p(\cdot)}(\mathbb{R}^{d},\mu)$ consists on those $\mu\_$measurable functions $f$ for which there exists $\lambda>0$ such that $\rho_{p(\cdot),\mu}\left(\frac{f}{\lambda}\right)<\infty,$ i.e.
\begin{equation*}
L^{p(\cdot)}(\mathbb{R}^{d},\mu) =\left\{f:\mathbb{R}^{d}\to \overline{\mathbb{R}}: f \; \text{is}\; \mu_{-} \text{measurable and} \; \rho_{p(\cdot),\mu}\left(\frac{f}{\lambda}\right)<\infty, \; \text{for some} \;\lambda>0\right\}
\end{equation*}
and the norm
\begin{equation}
\|f\|_{L^{p(\cdot)}(\mathbb{R}^{d},\mu)}=\inf\left\{\lambda>0:\rho_{p(\cdot),\mu}(f/\lambda)\leq 1\right\}.
\end{equation}
\end{defi}

It is well known that, if $p(\cdot) \in L H\left(\mathbb{R}^{d}\right)$ with $1<p_{-} \leq p^{+}<\infty$ the classical Hardy-Littlewood maximal function $\mathcal{M}$ is bounded on the variable Lebesgue space $L^{p(\cdot)}(\Bbb R^{d}),$ see \cite{dcruz1}. However, it is known that even though these are the sharpest possible point-wise conditions, they are not necessary. In \cite{LibroDenHarjHas} a necessary and sufficient condition is given for the $L^{p(\cdot)}$-boundedness of $\mathcal{M},$ but it is not an easy to work condition. The class $L H(\mathbb{R}^{d})$ is also sufficient for the boundedness on $L^{p(\cdot)}$-spaces of classical singular integrals of Calder\'on-Zygmund type, see \cite[Theorem 5.39]{dcruz}.\\

If $\mathcal{B}$ is a family of balls (or cubes) in $\mathbb{R}^{d}$, we say that $\mathcal {B}$ is $N$-finite if it has bounded overlappings for $N$, i.e., $\displaystyle\sum_{B\in\mathcal{B}}\chi_{B}(x)\leq N$ for all $x\in\mathbb {R}^{d}$; in other words, there is at most $N$ balls (resp. cubes) that intersect at the same time.\\

The following definition was introduced for the first time by Berezhno\v{\i} in \cite{Berez}, defined for a family of disjoint balls or cubes. In the context of variable spaces, it has been considered in \cite{LibroDenHarjHas}, allowing the family to have bounded overlappings.\\
\begin{defi}
Given an exponent $p(\cdot)\in\mathcal{P}(\mathbb{R}^{d})$, we will say that $p(\cdot)\in\mathcal{G}$, if for every family of balls (or cubes) $\mathcal{B}$ which is $N$-finite,
\begin{eqnarray*}
  \sum_{B\in\mathcal{B}}||f\chi_{B}||_{p(\cdot)}||g\chi_{B}||_{p'(\cdot)} &\lesssim& ||f||_{p(\cdot)}||g||_{p'(\cdot)}
\end{eqnarray*}
for all functions $f\in L^{p(\cdot)}(\mathbb{R}^{d})$ and $g\in L^{p'(\cdot)}(\mathbb{R}^{d})$. The constant only depends on N.
\end{defi}

\begin{teo}[Teorema 7.3.22 of \cite{LibroDenHarjHas}]\label{implication1}
If $p(\cdot)\in LH(\mathbb{R}^{d})$, then $p(\cdot)\in\mathcal{G}$
\end{teo}

We will consider only variable Lebesgue  spaces with respect to the Gaussian measure $\gamma_d,$ $L^{p(\cdot)}(\mathbb{R}^{d},\gamma_d).$  The next condition was introduced by E. Dalmasso and R. Scotto in \cite{DalSco}.

\begin{defi}\label{defipgamma}
Let $p(\cdot)\in\mathcal{P}(\mathbb{R}^{d},\gamma_{d})$, we say that $p(\cdot)\in\mathcal{P}_{\gamma_{d}}^{\infty}(\mathbb{R}^{d})$ if there exist constants $C_{\gamma_{d}}>0$ and $p_{\infty}\geq1$ such that
\begin{equation}
   |p(x)-p_{\infty}|\leq\frac{C_{\gamma_{d}}}{|x|^{2}},
\end{equation}
for $x\in\mathbb{R}^{d}\setminus\{(0,0,\ldots,0)\}.$
\end{defi}

\begin{obs}\label{obs4.1}
If $p(\cdot)\in\mathcal{P}_{\gamma_{d}}^{\infty}(\mathbb{R}^{d})$, then $p(\cdot)\in LH_{\infty}(\mathbb{R}^{d})$
\end{obs}
\begin{lemma}[Lemma 2.5 of \cite{DalSco}]\label{lemaequiPgamma}
If $1<p_{-}\leq p_{+}<\infty,$ the following statements are equivalent:
\begin{itemize}
  \item [(i)] $p(\cdot)\in\mathcal{P}_{\gamma_{d}}^{\infty}(\mathbb{R}^{d})$
  \item [(ii)] There exists $p_{\infty}>1$ such that
  \begin{eqnarray}
    C_{1}^{-1}\leq e^{-|x|^{2}(p(x)/p_{\infty}-1)}\leq C_{1} &\;\;\hbox{and}\;\;& C_{2}^{-1}\leq e^{-|x|^{2}(p'(x)/p'_{\infty}-1)}\leq C_{2},
  \end{eqnarray}
  for all $x\in\mathbb{R}^{d}$, where $C_{1}=e^{C_{\gamma_{d}}/p_{\infty}}$ and $C_{2}=e^{C_{\gamma_{d}}p'_{-}/p_{\infty}}$.
\end{itemize}
\end{lemma}

 Definition \ref{defipgamma} with Observation \ref{obs4.1} and Lemma \ref{lemaequiPgamma} end up strengthening the regularity conditions on the exponent functions $p(\cdot)$ to obtain the boundedness of the Ornstein-Uhlenbeck semigroup $\{T_{t}\}$, see \cite{MorPiUrb}. As a consequence of Theorem \ref{implication1}, we have
\begin{corollary}\label{solapamientoacotadoG}
If $p(\cdot)\in\mathcal{P}_{\gamma_{d}}^{\infty}(\mathbb{R}^{d})\cap LH_{0}(\mathbb{R}^{d})$, then $p(\cdot)\in\mathcal{G}.$
\end{corollary}

As we have already mentioned, the main result in this paper is the proof that the  Gaussian Riesz Potentials $I_\beta$, for $\beta\geq 1$, are bounded on Gaussian variable Lebesgue spaces under the same condition of regularity on \(p(\cdot)\) considered by Dalmasso and Scotto \cite{DalSco}.

\begin{teo}\label{boundLpvarRieszPot}
Let $p(\cdot)\in\mathcal{P}_{\gamma_{d}}^{\infty}(\mathbb{R}^{d})\cap LH_{0}(\mathbb{R}^{d})$ with $1<p_{-}\leq p_{+}<\infty$. Then for $\beta\geq 1$ there exists a constant $C>0,$ depending only on $p$, $\beta$ and the dimension $d$ such that 
\begin{equation}
\| I_\beta f\|_{p(\cdot),\gamma_{d}} \leq C \| f\|_{p(\cdot),\gamma_{d}},
\end{equation}
for any $f \in L^{p(\cdot)}(\gamma_d).$ \\
 \end{teo}

Trivially, Theorem \ref{boundLpvarRieszPot} give us an alternative proof of the boundedness of the Gaussian Riesz Potentials $I_\beta$, for $\beta\geq 1$ on Gaussian Lebesgue spaces $L^p(\gamma_d)$, by simply taking the exponent function constant, but the constant $C$ depends on $\beta$ and the dimension, which is weaker than the estimate obtained using Meyer's multiplier theorem mentioned above.
\section{Proof of the main result.}

In order to prove our main result, Theorem \ref{boundLpvarRieszPot} we need some technical results.
\begin{lemma}\label{lema3.26CU}
Let $\rho(\cdot):\mathbb{R}^{d}\rightarrow[0,\infty)$ be such that $\rho(\cdot)\in LH_{\infty}(\mathbb{R}^{d})$, $0<\rho_{\infty}<\infty$, and let $R(x)=(e+|x|)^{-N}$, $N>d/\rho_{-}$. Then there exists a constant $C$ depending on $d$, $N$ and the $LH_{\infty}$ constant of $\rho(\cdot)$ such that given any set $E$ and
any function $F$ with $0\leq F(y)\leq 1$, for all $y\in E$,
\begin{eqnarray}
  \int_{E}F^{\rho(y)}(y)dy &\leq& C\int_{E}F^{\rho_{\infty}}(y)dy + \int_{E}R^{\rho_{-}}(y)dy,\label{3.26.1} \\
  \int_{E}F^{\rho_{\infty}}(y)dy &\leq& C\int_{E}F^{\rho(y)}(y)dy + \int_{E}R^{\rho^{-}}(y)dy.\label{3.26.2}
\end{eqnarray}
\end{lemma}
For the proof see Lemma 3.26 of \cite{dcruz}.			
\begin{lemma}\label{lemgamma}
If $\alpha>0$, there exists a constant $C>0$ such that
\begin{equation}\label{desigualdadgamma}
\int_{0}^{1}\left(-log(\sqrt{1-u})\right)^{\alpha-1}du =
C\Gamma(\alpha)<\infty
\end{equation}
\end{lemma}
\begin{proof}
Taking the change of variables $t=-log(\sqrt{1-u})$ then $u=1-e^{-2t}$\\
and $du=2e^{-2t}dt$. For $\alpha> 0$ we get
\begin{equation*}
\int_{0}^{1}\left(-log(\sqrt{1-u})\right)^{\alpha-1}du =
2\int_{0}^{+\infty}t^{\alpha-1}e^{-2t}dt=C\Gamma(\alpha)<\infty
\end{equation*}
\end{proof}

\begin{lemma}\label{leminteg}
For $\beta>0$\\
\begin{enumerate}
\item[ i)]
\begin{equation}
\int_{1/2}^{1}\left(-log(\sqrt{1-u})\right)^{\frac{\beta}{2}}\frac{1}{\sqrt{1-u}}du\,<\infty.
\end{equation}
\item[ ii)]
\begin{equation}\label{gammaint}
\int_{0}^{1/2}\frac{\left(-log(\sqrt{1-u})\right)^{\frac{\beta}{2}}}{u}du\,<\infty.
\end{equation}
\end{enumerate}
\end{lemma}
\begin{proof}\quad\\
\begin{enumerate}
\item[ i)] Using H\"older's inequality with $p=\frac{3}{2}$, $q=3$ and  Lemma \ref{lemgamma}, with $\alpha=\frac{3\beta}{2}+1$, we have that
\begin{eqnarray*}
\int_{1/2}^{1}\left(-log(\sqrt{1-u})\right)^{\frac{\beta}{2}}\frac{du}{\sqrt{1-u}}&\leq
&\int_{0}^{1}\left(-log(\sqrt{1-u})\right)^{\frac{\beta}{2}}\frac{1}{\sqrt{1-u}}du\\
&\leq
&\left(\int_{0}^{1}\left(-log(\sqrt{1-u})\right)^{\frac{3\beta}{2}}du\right)^{1/3}
\left(\int_{0}^{1}\frac{du}{(1-u)^{3/4}}\right)^{2/3}\\
&=
&\left(\int_{0}^{1}\left(-log(\sqrt{1-u})\right)^{(\frac{3\beta}{2}+1)-1}du\right)^{1/3}
\left(\int_{0}^{1}\frac{du}{(1-u)^{3/4}}\right)^{2/3}\,<\infty
\end{eqnarray*}
\item[ii)] Let us rewrite the integral as
\begin{equation*}
\int_{0}^{1/2}\frac{\left(-log(\sqrt{1-u})\right)^{\frac{\beta}{2}}}{u}du=\int_{0}^{1/2}\frac{\left(-log(\sqrt{1-u})\right)}{u}\left(-log(\sqrt{1-u})\right)^{\frac{\beta}{2}-1}du,
\end{equation*}
since $\displaystyle\lim_{u\to
0}\frac{\left(-log(\sqrt{1-u})\right)}{u}=\frac{1}{2}$  and
$\displaystyle\frac{\left(-log(\sqrt{1-u})\right)}{u}$ is bounded in $(0,1/2],$ then we have  we have by (\ref{desigualdadgamma})
\begin{eqnarray*}
\int_{0}^{1/2}\frac{\left(-log(\sqrt{1-u})\right)^{\frac{\beta}{2}}}{u}du&=& \int_{0}^{1/2}\left(-log(\sqrt{1-u})\right)^{\frac{\beta}{2}-1}\frac{\left(-log(\sqrt{1-u})\right)}{u}du\\
&\leq&
C\int_{0}^{1/2}\left(-log(\sqrt{1-u})\right)^{\frac{\beta}{2}-1}du\,<
\infty.
\end{eqnarray*}
\end{enumerate}
\end{proof}
We are now ready to prove the main result, Theorem \ref{boundLpvarRieszPot}.

\begin{proof}

As usual, we split the operator $I_\beta$ in its local part and its global part
\begin{equation*}
I_{\beta}f(x)=I_{\beta,L}f(x)+I_{\beta,G}f(x),
\end{equation*}
where
\begin{equation*}
I_{\beta,L}f(x)=I_{\beta}(f\chi_{B_{h}(\cdot)} )(x)
\end{equation*}
is the local part,
\begin{equation*}
I_{\beta,G}f(x)=I_{\beta}(f\chi_{B^{c}_{h}(\cdot)} )(x)
\end{equation*}
is the global part, and for $x\in \Bbb R^{d}$ by taking $m(x)=1\wedge \frac{1}{|x|}$,\\

$B_{h}(x):=\lbrace y\in \Bbb R^{d}:|x-y|<dm(x)\rbrace$ is an $hiperbolic$ $ball$ ($admissible$ $ball$).\\

Let us take
$\omega(s)=\displaystyle\frac{e^{-\frac{|y-e^{-s}x|^{2}}{1-e^{-2s}}}}{(1-e^{-2s})^{\frac{d}{2}}}$,
then, from (\ref{kernelRiesz})
\begin{eqnarray*}
N_{\beta/2}(x,y)&=&\frac{1}{\pi^{\frac{d}{2}}\Gamma(\frac{\beta}{2})}\int_{0}^{+\infty}
t^{\frac{\beta}{2}-1}\left(\omega(t)-\omega(+\infty)\right)dt\\
&=&\frac{1}{\pi^{\frac{d}{2}}\Gamma(\frac{\beta}{2})}\int_{0}^{+\infty}
t^{\frac{\beta}{2}-1}\left(-\int_{t}^{+\infty}\frac{\partial
\omega(s)}{\partial s}ds\right)dt.
\end{eqnarray*}
Thus, using Hardy's inequality, see \cite{st1}
$$
\left|N_{\beta/2}(x,y)\right|\leq \frac{1}{\pi^{\frac{d}{2}}\Gamma(\frac{\beta}{2})}\int_{0}^{+\infty}
t^{\frac{\beta}{2}-1}\int_{t}^{+\infty}\left|\frac{\partial
\omega(s)}{\partial s}\right|ds\;dt \leq
\frac{1}{\pi^{\frac{d}{2}}\Gamma(\frac{\beta}{2})}\frac{2}{\beta}\int_{0}^{+\infty}
s^{\frac{\beta}{2}}\left|\frac{\partial
\omega(s)}{\partial s}\right|ds.
$$
Now,
\begin{eqnarray*}
\frac{\partial \omega(s)}{\partial
s}&=&\frac{(1-e^{-2s})^{\frac{d}{2}}e^{-\frac{|y-e^{-s}x|^{2}}{1-e^{-2s}}}}{(1-e^{-2s})^{d}}
\left(\frac{-2(1-e^{-2s})(y-e^{-s}x)\cdot(e^{-s}x)+|y-e^{-s}x|^{2}e^{-2s}}{(1-e^{-2s})^{2}}\right)\\
&&\hspace{4.5cm}-\frac{1}{(1-e^{-2s})^{d}}\left(\frac{d}{2}e^{\frac{|y-e^{-s}x|^{2}}{1-e^{-2s}}}(1-e^{-2s})^{\frac{d}{2}-1}2e^{-2s}\right)\\
&=&\omega(s)\left(\frac{-2(1-e^{-2s})(y-e^{-s}x)\cdot(e^{-s}x)+|y-e^{-s}x|^{2}e^{-2s}}{(1-e^{-2s})^{2}}-\frac{de^{-2s}}{(1-e^{-2s})}\right).
\end{eqnarray*}
Then, taking $u=1-e^{-2s},\; du=2e^{-2s}ds,$ i.e., $e^{-s} = \sqrt{1-u},$ we have
\begin{eqnarray*}
\left|N_{\beta/2}(x,y)\right|&\leq
&\int_{0}^{1}\left(-log(\sqrt{1-u})\right)^{\frac{\beta}{2}}\frac{e^{-\frac{|y-\sqrt{1-u}x|^{2}}{u}}}{u^{\frac{d}{2}}}
\\
&&\hspace{0.2cm}\times\left(\frac{2u|y-\sqrt{1-u}x|\sqrt{1-u}|x|+|y-\sqrt{1-u}x|^{2}(1-u)}{u^{2}}+\frac{d(1-u)}{u}\right)\frac{du}{2(1-u)}\\
&=&\int_{0}^{1}\left(-log(\sqrt{1-u})\right)^{\frac{\beta}{2}}\frac{e^{-\frac{|y-\sqrt{1-u}x|^{2}}{u}}}{u^{\frac{d}{2}}}
\left(\frac{|y-\sqrt{1-u}x||x|}{u\sqrt{1-u}}+\frac{|y-\sqrt{1-u}x|^{2}}{2u^{2}}+\frac{d}{2u}\right)du\\
&=& \int_{0}^{1}\frac{\left(-log(\sqrt{1-u})\right)^{\frac{\beta}{2}}}{u^{\frac{d}{2}+1}}e^{-\frac{|y-\sqrt{1-u}x|^{2}}{u}}
\frac{|y-\sqrt{1-u}x||x|}{\sqrt{1-u}}du \\
&& \hspace{1cm}+\int_{0}^{1}\left(-log(\sqrt{1-u})\right)^{\frac{\beta}{2}}\frac{e^{-\frac{|y-\sqrt{1-u}x|^{2}}{u}}}{u^{\frac{d}{2}+1}}\frac{|y-\sqrt{1-u}x|^{2}}{2u}du\\
&&  \hspace{1.5cm}+\int_{0}^{1}\left(-log(\sqrt{1-u})\right)^{\frac{\beta}{2}}\frac{e^{-\frac{|y-\sqrt{1-u}x|^{2}}{u}}}{u^{\frac{d}{2}}}\frac{d}{2u} du\\
&&=I+II+III,
\end{eqnarray*}
where
\begin{equation}
I = \int_{0}^{1}\frac{\left(-log(\sqrt{1-u})\right)^{\frac{\beta}{2}}}{u^{\frac{d}{2}+1}}e^{-\frac{|y-\sqrt{1-u}x|^{2}}{u}}
\frac{|y-\sqrt{1-u}x||x|}{\sqrt{1-u}}du,
\end{equation}
\begin{equation}
II= \int_{0}^{1}\left(-log(\sqrt{1-u})\right)^{\frac{\beta}{2}}\frac{e^{-\frac{|y-\sqrt{1-u}x|^{2}}{u}}}{u^{\frac{d}{2}+1}}\frac{|y-\sqrt{1-u}x|^{2}}{2u}du,
\end{equation}
and
\begin{equation}
III = \int_{0}^{1}\left(-log(\sqrt{1-u})\right)^{\frac{\beta}{2}}\frac{e^{-\frac{|y-\sqrt{1-u}x|^{2}}{u}}}{u^{\frac{d}{2}}}\frac{d}{2u} du.
\end{equation}

\begin{itemize}
\item Let us study the local part first. We need to bound each of the terms $I$, $II$ and $III$ in this part.

For $I$, since we are in the local part 
$|x-y| \leq \frac{d}{|x|}$, then  we have $|x-y||x| \leq d $  therefore,
\begin{eqnarray}\label{ineqlocpart}
\nonumber |y-\sqrt{1-u}\,x|^{2} & \geq& (|y-x|-|x|(1-\sqrt{1-u}))^{2} \\
& \geq& |y-x|^{2}-2|x||y-x| \frac{u}{1+\sqrt{1-u}} \geq|y-x|^{2}-2 d \; u.
\end{eqnarray}
On the other hand, it is well known that, there exist $C>0$ such that\\
for any $x>0$, $\alpha \geq 0$ and $c>0$,
\begin{equation}\label{expineq}
x^\alpha e^{-cx^2} \leq C \hspace{0.5 cm}\cdot
\end{equation}
 Thus, using (\ref{ineqlocpart}) and   (\ref{expineq}) twice, with $\alpha =1,\; c= \frac12$ and $\alpha=d-1, c=\frac12$, we get
\begin{eqnarray*}
I&=
&\int_{0}^{1}\frac{\left(-log(\sqrt{1-u})\right)^{\frac{\beta}{2}}}{u^{\frac{d}{2}}}e^{-\frac{|y-\sqrt{1-u}x|^{2}}{2u}}
\left(\left|\frac{y-\sqrt{1-u}x}{\sqrt{u}} 
\right|e^{-\frac{|y-\sqrt{1-u}x|^{2}}{2u}}\right)\frac{|x|}{\sqrt{u}\sqrt{1-u}}du\\
&\leq
&C|x|\int_{0}^{1}\left(-log(\sqrt{1-u})\right)^{\frac{\beta}{2}}\frac{e^{-\frac{|y-x|^{2}}{2u}}}{u^{\frac{d}{2}}}
\frac{du}{\sqrt{u}\sqrt{1-u}}
\end{eqnarray*}
\begin{eqnarray*}
&=&
C|x|\int_{0}^{1}\left(-log(\sqrt{1-u})\right)^{\frac{\beta}{2}}\frac{e^{-\frac{|y-x|^{2}}{2u}}}{u^{\frac{d-1}{2}}}
\frac{du}{u\sqrt{1-u}}\\
&\leq
&\frac{C|x|}{|x-y|^{d-1}}\int_{0}^{1}\left(-log(\sqrt{1-u})\right)^{\frac{\beta}{2}}
\frac{du}{u\sqrt{1-u}}.\\
\end{eqnarray*}
By Lemma \ref{leminteg} we get
\begin{eqnarray*}
&&\int_{0}^{1}\left(-log(\sqrt{1-u})\right)^{\frac{\beta}{2}}
\frac{du}{u\sqrt{1-u}}\\
&& \hspace{1cm} = \int_{0}^{1/2}\left(-log(\sqrt{1-u})\right)^{\frac{\beta}{2}}
\frac{du}{u\sqrt{1-u}}\; +
\;\int_{1/2}^{1}\left(-log(\sqrt{1-u})\right)^{\frac{\beta}{2}}
\frac{du}{u\sqrt{1-u}}\\
&& \hspace{1cm} \leq C\int_{0}^{1/2}\frac{\left(-log(\sqrt{1-u})\right)^{\frac{\beta}{2}}}{u}
du\; +
\;C\int_{1/2}^{1}\left(-log(\sqrt{1-u})\right)^{\frac{\beta}{2}}
\frac{du}{\sqrt{1-u}}<\infty,
\end{eqnarray*}
since  $\displaystyle\frac{1}{u}$ in bounded on $[1/2,1]$ and 
$\displaystyle\frac{1}{\sqrt{1-u}}$ is bounded on  $[0,1/2]$. Thus, we have
\begin{equation*}
I\leq \frac{C|x|}{|x-y|^{d-1}}.
\end{equation*}
For  $II$,  we use again (\ref{ineqlocpart}) and (\ref{expineq}) with $\alpha =2$ and $c=\frac12$ we have
\begin{eqnarray*}
II &=&
\int_{0}^{1}\left(-log(\sqrt{1-u})\right)^{\frac{\beta}{2}}\frac{e^{-\frac{|y-\sqrt{1-u}x|^{2}}{u}}}{u^{\frac{d}{2}+1}}\frac{|y-\sqrt{1-u}x|^{2}}{2u}du\\
&\leq &
C\int_{0}^{1}\left(-log(\sqrt{1-u})\right)^{\frac{\beta}{2}}\frac{e^{-\frac{|y-\sqrt{1-u}x|^{2}}{2u}}}{u^{\frac{d}{2}+1}}du\\
&\leq
& C\int_{0}^{1}\left(-log(\sqrt{1-u})\right)^{\frac{\beta}{2}}\frac{e^{-\frac{|y-x|^{2}}{2u}}}{u^{\frac{d}{2}+1}}du\\
&=&C \int_{0}^{1/2}\frac{\left(-log(\sqrt{1-u})\right)}{u}\left(-log(\sqrt{1-u})\right)^{\frac{\beta}{2}-1}\frac{e^{-\frac{|y-x|^{2}}{2u}}}{u^{\frac{d}{2}}}du\;+
\;\int_{1/2}^{1}\left(-log(\sqrt{1-u})\right)^{\frac{\beta}{2}}\frac{e^{-\frac{|y-x|^{2}}{2u}}}{u^{\frac{d}{2}+1}}du.\\
\end{eqnarray*}
Since $\displaystyle\lim_{u\to
0}\frac{\left(-log(\sqrt{1-u})\right)}{u}=\frac{1}{2}$, this function is bounded on $[0,1/2]$ and 
$\displaystyle\frac{1}{u}$ is bounded on $[1/2,1]$, thus we get
\begin{eqnarray*}
II&\leq
&C\int_{0}^{1/2}\left(-log(\sqrt{1-u})\right)^{\frac{\beta}{2}-1}\frac{e^{-\frac{|y-x|^{2}}{2u}}}{u^{\frac{d}{2}}}du\;+
\;C\int_{1/2}^{1}\left(-log(\sqrt{1-u})\right)^{\frac{\beta}{2}}\frac{e^{-\frac{|y-x|^{2}}{2u}}}{u^{\frac{d}{2}}}du\\
&\leq
&C\int_{0}^{1}\left(-log(\sqrt{1-u})\right)^{\frac{\beta}{2}-1}\frac{e^{-\frac{|y-x|^{2}}{2u}}}{u^{\frac{d}{2}}}du\;+
\;C\int_{0}^{1}\left(-log(\sqrt{1-u})\right)^{\frac{\beta}{2}}\frac{e^{-\frac{|y-x|^{2}}{2u}}}{u^{\frac{d}{2}}}du\\
&=&C\mathcal{K}_{2}(x-y)\;+\;CG_{2}(x-y),
\end{eqnarray*}
where
$$ \mathcal{K}_2(x) := \int_{0}^{1}\left(-log(\sqrt{1-u})\right)^{\frac{\beta}{2}-1}\frac{e^{-\frac{|x|^{2}}{2u}}}{u^{\frac{d}{2}}}du,$$
and 
$$
G_{2}(x):=\int_{0}^{1}\left(-log(\sqrt{1-u})\right)^{\frac{\beta}{2}}\frac{e^{-\frac{|x|^{2}}{2u}}}{u^{\frac{d-1}{2}}}\frac{du}{\sqrt{u}}.
$$
Again by (\ref{expineq})
\begin{eqnarray*}
G_{2}(x-y)&=&\int_{0}^{1}\left(-log(\sqrt{1-u})\right)^{\frac{\beta}{2}}\frac{e^{-\frac{|y-x|^{2}}{2u}}}{u^{\frac{d-1}{2}}}\frac{du}{\sqrt{u}}\\
&\leq& \frac{C}{|x-y|^{d-1}}\int_{0}^{1}\left(-log(\sqrt{1-u})\right)^{\frac{\beta}{2}}\frac{du}{\sqrt{u}}.
\end{eqnarray*}
Now, by H\"older's inequality with $p=\frac{3}{2}$ y $q=3$ we have
\begin{equation*}
\int_{0}^{1}\left(-log(\sqrt{1-u})\right)^{\frac{\beta}{2}}\frac{du}{\sqrt{u}}\leq
\left(\int_{0}^{1}\left(-log(\sqrt{1-u})\right)^{\frac{3\beta}{2}}du\right)^{1/3}
\left(\int_{0}^{1}\frac{du}{u^{3/4}}\right)^{2/3}.
\end{equation*}
Thus, by Lemma (\ref{lemgamma}) we obtain
\begin{equation*}
\int_{0}^{1}\left(-log(\sqrt{1-u})\right)^{\frac{3\beta}{2}}du=\int_{0}^{1}\left(-log(\sqrt{1-u})\right)^{(\frac{3\beta}{2}+1)-1}du\;<\;\infty
\end{equation*}
and trivially
$\displaystyle{\int_{0}^{1}\frac{du}{u^{3/4}}}< \infty.$\\

Finally for $III$, by analogous arguments as in $II$, we get
$$
III = \int_{0}^{1}\left(-log(\sqrt{1-u})\right)^{\frac{\beta}{2}}\frac{e^{-\frac{|y-\sqrt{1-u}x|^{2}}{u}}}{u^{\frac{d}{2}}}\frac{d}{2u} du\leq C \mathcal{K}_{2}(x-y)\;+\;CG_{2}(x-y).
$$

Therefore, 
\begin{eqnarray*}
\left|N_{\beta/2}(x,y)\right|&\leq & I+II+III\\
&\leq & \frac{C|x|}{|x-y|^{d-1}}\,+\,C\mathcal{K}_{2}(x-y)\,+\,\frac{C}{|x-y|^{d-1}}\\
&\leq & C\frac{|x|+1}{|x-y|^{d-1}}\,+\,C\mathcal{K}_{2}(x-y)= C\mathcal{K}_{3}(x,y)\,+\,C\mathcal{K}_{2}(x-y),
\end{eqnarray*}
where 
$$ \mathcal{K}_{3}(x,y) :=  \frac{|x|+1}{|x-y|^{d-1}}.$$
Thus, using the above estimates, we conclude that  the local part
$I_{\beta,L}$ can be bounded as

\begin{eqnarray*}
|I_{\beta,L} f(x)| &=&  |I_{\beta} (f\chi_{B_h(\cdot)})(x)|  =  \Big| \int_{B_h(x)}  N_{\beta/2} (x,y)  f(y) \;dy \Big|\\
&\lesssim& \int_{B_h(x)}  {\mathcal K}_3(x,y)  |f(y)| \;dy   + \int_{B_h(x)}  {\mathcal K}_2(x-y) | f(y) |\;dy \\
&=& IV + V.\\
\end{eqnarray*}

Now, to bound IV and V, we need to take a countable family of admissible balls $ \mathcal{F}$ that satisfies the condition of Lemma 4.3 of \cite{urbina2019}. In particular, \(\mathcal{F}\) verifies
\begin{enumerate}
\item[i)] For each $B \in\mathcal{F}$ let  \(\tilde{B}=2 B,\) then, the family of those balls $\tilde{\mathcal{F}}= \{B(0,1),\{\tilde{B}\}_{B \in \mathcal{F}}\}$ is a covering of \(\mathbb{R}^{d}\);
\item[ii)] \(\mathcal{F}\) has a bounded overlaps property;
\item[iii)] Every ball \(B \in \mathcal{F}\) is contained in an admissible ball, and therefore for
any pair \(x, y \in B, e^{-|x|^{2}} \sim e^{-|y|^{2}}\) with constants independent of \(B\)
\item[iv)]  There exists a uniform positive constant \(C_{d}\) such that, if \(x \in B \in \mathcal{F}\) then \(B_h(x) \subset C_{d} B:=\hat{B} .\) Moreover, the collection \(\hat{\mathcal{F}}=\{\hat{B}\}_{B \in \mathcal{F}}\) also satisfies the properties ii) and iii).\\
\end{enumerate}

Given $B \in  \mathcal{F},$ if  $x \in B$ then $B_h(x) \subset
\hat{B}$, we get,
\begin{eqnarray*}
IV &=&  (1+|x|) \sum_{k=0}^\infty \int_{2^{-(k+1)}C_d m(x) < |x-y| < 2^{-k}C_d m(x)} \frac{|f(y)| \chi_ {\hat{B}}}{|x-y|^{d-1}} dy\\
&\leq& C_d 2^d \mathcal{M}(f\chi_{\hat{B}})(x) (1+|x|) m(x)
\sum_{k=0}^\infty 2^{-(k+1)} \leq C
\mathcal{M}(f\chi_{\hat{B}})(x)\chi_{B}(x),
\end{eqnarray*}
where $\mathcal{M}(g)$ is the classical Hardy-Littlewood maximal
function of the function $g.$\\

 On the other hand, let us consider the function 
$\varphi(y) =\frac{1}{\pi^{d/2}} e^{-\frac{1}{2}|y|^2},$ 
then $\displaystyle\int_{\mathbb{R}^d} \varphi(y) dy =1$. It is well known that
$\varphi$ is a non-increasing radial function, and given $t>0,\,$
we rescale this function as $\varphi_{\sqrt{t}} (y) = t^{-d/2}
\varphi(y/\sqrt{t}),$ and, since $0\leq \varphi \in
L^1(\mathbb{R}^d),$ $\left\{\varphi_{\sqrt{t}}\right \}_{t>0}\;$ is the 
classical (Gauss-Weiertrass) approximation of the identity in $\mathbb{R}^d.\,$ Then,
since\\
$$\displaystyle\int_{0}^{1}\left(-log(\sqrt{1-t})\right)^{\frac{\beta}{2}-1}
dt <\infty,$$
we get
\begin{eqnarray*}
V &=& \int_{B_h(x)}  \mathcal{K}_2(x-y) | f(y) |\;dy  =
\int_{B_h(x)}  \Big(\int_0^1 \varphi_{\sqrt{t}}(x-y)
\left(-log(\sqrt{1-t})\right)^{\frac{\beta}{2}-1}
dt\Big) |f(y)| dy\\
&\leq&  \int_{B_h(x)} \Big(\sup_{t>0}  \varphi_{\sqrt{t}}(x-y)
\Big)
\Big(\int_0^1\left(-log(\sqrt{1-t})\right)^{\frac{\beta}{2}-1}
dt\Big) |f(y)| dy\\
&\leq& C \int_{B_h(x)} \Big(\sup_{t>0}  \varphi_{\sqrt{t}}(x-y)
\Big) |f(y)| dy.
\end{eqnarray*}

Again, using the family  $\mathcal{F}$, if  $x \in B$ then $B_h(x)
\subset \hat{B}$. By a similar argument as before and as in as result in Stein's book \cite[Chapter II \S4, Theorem 4]{st1},
\begin{eqnarray*}
V &=& \int_{B_h(x)}  \mathcal{K}_2(x-y) | f(y) |\;dy  \leq C
\int_{\mathbb{R}^d} \left(\sup_{t>0}  \varphi_{\sqrt{t}}(x-y) \right)
|f(y)| \chi_{\hat{B}}(y) dy\\
 &\lesssim&  \sum_{B \in \mathcal{F}}   \left| \left(  \sup_{t>0}\varphi_{\sqrt{t}} *  |f \chi_{\hat{B}}|\right)(x) \right| \chi_{B}(x)\leq  \sum_{B \in \mathcal{F}} \mathcal{M}(f \chi_{\hat{B}})(x)
\chi_{B}(x),
\end{eqnarray*}
which yields,
$|I_{\beta,L}f(x)|\leq\displaystyle\sum_{B \in \mathcal{F}} \mathcal{M}(f \chi_{\hat{B}})(x)
\chi_{B}(x).$\\
Then, for \(f \in L^{p(\cdot)}\left(\mathbb{R}^{d}, \gamma_d\right)\) we will use the characterization of the norm by duality, we get
\begin{equation}\label{ineq2}
\left\|I_{\beta,L}(f)\right\|_{p(\cdot), \gamma_{d}} \leq 2 \sup _{\|g\|_{p^{\prime}(\cdot), \gamma_{d}} \leq 1} \int_{\mathbb{R}^{d}}\left|I_{\beta,L}f(x)\right||g(x)| \gamma_d(dx).
\end{equation}

Using the estimates above, we get
\begin{eqnarray*}
\int_{\mathbb{R}^{d}}\left|I_{\beta,L}f(x)\right||g(x)| \gamma_d(dx) &\lesssim& 
\sum_{B \in \mathcal{F}} \int_{B} \mathcal{M}\left(f \chi_{\hat{B}(\cdot)}\right)(x)|g(x)| e^{-|x|^{2}} d x \\
&&  \approx\sum_{B \in \mathcal{F}} e^{-\left|c_{B}\right|^{2}} \int_{B} \mathcal{M}\left(f \chi_{\hat{B}(\cdot)}\right)(x)|g(x)| d x,
\end{eqnarray*}
where \(c_{B}\) is the center of \(B\) and \(\hat{B}\) and we have used property iii) above, i.e. that over each ball of the family \(\mathcal{F},\) the values of \(\gamma_d\) are all equivalent.\\ Applying H\"older's inequality for \(p(\cdot)\) and \(p^{\prime}(\cdot)\) with respect of the Lebesgue measure and the boundedness of \(\mathcal{M}\) on $L^{p(\cdot)}\left(\mathbb{R}^{d}\right),$ we get
\begin{eqnarray}\label{ineq3}
\nonumber \int_{\mathbb{R}^{d}}\left|I_{\beta,L}f(x)\right||g(x)| \gamma_{d}(d x) &\lesssim & \sum_{B \in \mathcal{F}} e^{-\left|c_{B}\right|^{2}}\left\|\mathcal{M}\left(f \chi_{\hat{B}(\cdot)}\right) \chi_{B}\right\|_{p(\cdot)}\left\|g \chi_{B}\right\|_{p^{\prime}(\cdot)} \\
\nonumber &\lesssim &  \sum_{B \in \mathcal{F}} \mathrm{e}^{-\left|c_{B}\right|^{2}}\left\|f \chi_{\hat{B}}\right\|_{p(\cdot)}\left\|g \chi_{B}\right\|_{p^{\prime}(\cdot)} \\
&=& \sum_{B \in \mathcal{F}} \mathrm{e}^{-\left|c_{B}\right|^{2} / p_{\infty}}\left\|f \chi_{\hat{B}}\right\|_{p(\cdot)} e^{-\left|c_{B}\right|^{2} / p_{\infty}^{\prime}}\left\|g \chi_{\hat{B}}\right\|_{p^{\prime}(\cdot)}.
\end{eqnarray}
Since \(p \in P_{\gamma_{d}}^{\infty}\left(\mathbb{R}^{d}\right)\) and \(p_{-}>1$ then $p'\in P_{\gamma_{d}}^{\infty}\left(\mathbb{R}^{d}\right) .\) Thus, from Lemma \(1.4,\) for
every \(x \in \mathbb{R}^{d}\)
\begin{equation}\label{equiv}
e^{-|x|^{2}\left(p(x) / p_{\infty}-1\right)} \leq C_{1} \text { and } e^{-|x|^{2}\left(p^{\prime}(x) / p_{\infty}^{\prime}-1\right)} \leq C_{2}.
\end{equation}

Moreover, since the values of the Gaussian measure \(\gamma_{d}\) are all equivalent on any ball $\hat{B} \in \tilde{\mathcal{F}}$, we have
\begin{eqnarray*}
\int_{\hat{B}}\left(\frac{|f(y)|}{e^{\left|c_{B}\right|^{2} / p_{\infty}}\left\|f \chi_{\hat{B}}\right\|_{p(\cdot), \gamma_{d}}}\right)^{p(y)} d y &\lesssim & \int_{\hat{B}}\left(\frac{|f(y)|}{\left\|f \chi_{\hat{B}}\right\|_{p(\cdot), \gamma_{d}}}\right)^{p(y)} e^{-|y|^{2}\left(p(y) / p_{\infty}-1\right)} \gamma_{d}(dy) \\
&\lesssim &\int_{\hat{B}}\left(\frac{|f(y)|}{\left\|f \chi_{\hat{B}}\right\|_{p(\cdot), \gamma_{d}}}\right)^{p(y)} \gamma_{d}(dy) \lesssim  1,
\end{eqnarray*}
which yields
$$
e^{-\left|c_{B}\right|^{2} / p_{\infty}}\left\|f \chi_{\hat{B}}\right\|_{p(\cdot)} \lesssim \left\|f \chi_{\hat{B}}\right\|_{p(\cdot), \gamma_{d}}.
$$
Similarly, by the second inequality of (\ref{equiv}) we also get
$$
e^{-\left|c_{B}\right|^{2} / p_{\infty}^{\prime}}\left\|g \chi_{\hat{B}}\right\|_{p^{\prime}(\cdot)} \lesssim \left\|g \chi_{\hat{B}}\right\|_{p^{\prime}(\cdot), \gamma_{d}}.
$$

Replacing both estimates in (\ref{ineq3}) we obtain
\begin{eqnarray*}
\int_{\mathbb{R}^{d}}\left|I_{\beta,L}f(x)\right||g(x)| \gamma_{d}(d x) &  \lesssim  &\sum_{B \in \mathcal{F}}\left\|f \chi_{\hat{B}}\right\|_{p(\cdot), \gamma_{d}}\left\|g \chi_{\hat{B}}\right\|_{p^{\prime}(\cdot), \gamma_{d}} \\
&=& \sum_{B \in \mathcal{F}}\left\|f \chi_{\hat{B}} e^{-|\cdot|^{2} / p(\cdot) |}\right\|_{p(\cdot)}\left\|g \chi_{\hat{B}} e^{-|\cdot|^{2} / p^{\prime}(\cdot)}\right\|_{p^{\prime}(\cdot)}.
\end{eqnarray*}
Since the family of balls \(\hat{\mathcal{F}}\) has bounded overlaps, from Corollary 1.1 applied to \(f e^{-|\cdot|^{2} / p(\cdot)} \in L^{p(\cdot)}(\mathbb{R}^{d})\) and \(g e^{-|\cdot|^{2} / p^{\prime}(\cdot)} \in L^{p^{\prime}(\cdot)}(\mathbb{R}^{d}),\) it follows that
$$
\int_{\mathbb{R}^{d}}\left|I_{\beta,L}f(x)\right||g(x)| \gamma_{d}(d x) \leqslant\|f\|_{p(\cdot), \gamma_{d}}\|g\|_{p^{\prime}(\cdot), \gamma_{d}}.
$$
Taking supremum over all functions \(g\) with \(\|g\|_{p^{\prime}(\cdot), \gamma_{d}} \leq 1,\) from (\ref{ineq2}) we get finally
$$
\left\|I_{\beta,L}(f)\right\|_{p(\cdot), \gamma_{d}}\leq C\|f\|_{p(\cdot), \gamma_{d}}.\\
$$
Therefore,  the local part $I_{\beta,L}$ is bounded in $L^{p(\cdot)}(\gamma_d)$.\\

\item Now, let us study the global part. Again, since 
$$
\left|N_{\beta/2}(x,y)\right|\leq
I+II+III,
$$
we need to estimate each term in this part. As usual for the global part, the arguments are completely different but are based on the following technical result, obtained S. P\'erez \cite[Lemma 3.1]{pe}, see also \cite[\S4.5]{urbina2019}.To simplify the notation, in what follows we denote
$$
a=a(x, y):=|x|^{2}+|y|^{2}, b=b(x, y):=2\langle x, y\rangle \text {, }
$$
$$
u(t)=u(t ; x, y):=\frac{|y-\sqrt{1-t x}|^{2}}{t}=\frac{a}{t}-\frac{\sqrt{1-t}}{t} b-|x|^{2},
$$
$$t_{0}=2 \frac{\sqrt{a^{2}-b^{2}}}{a+\sqrt{a^{2}-b^{2}}},$$
 then
$$
u\left(t_{0}\right)=\frac{\sqrt{a^{2}-b^{2}}}{2}+\frac{a}{2}-|x|^{2}=\frac{|y|^{2}-|x|^{2}}{2}+\frac{\sqrt{a^{2}-b^{2}}}{2}
$$
and
$$
t_{0} \sim \frac{\sqrt{a^{2}-b^{2}}}{a} \sim \frac{\sqrt{a-b}}{\sqrt{a+b}}=\frac{|x-y|}{|x+y|}.
$$
It is well known that $t_{0}<1$, the minimum of $u(t)$ is attained at 
$t_{0}$ and 
$$\frac{1}{t_{0}^{d/2}}\lesssim |x+y|^{d}.$$
For details and other properties of these terms, see \cite{pe}, \cite{TesSon} or \cite{urbina2019}.\\

Let us fix $x\in \Bbb R^{d}$ and consider $E_{x}=\lbrace y\in \Bbb R^{d}:b>0\rbrace$.\\
\begin{itemize}
\item \underline{Case $b\leq 0$}:\\

First, for $0< \epsilon <1$,  using inequality (\ref{expineq}) we have
\begin{eqnarray*}
I&=&
\int_{0}^{1}\left(-log(\sqrt{1-t})\right)^{\frac{\beta}{2}}\frac{e^{-u(t)}}{t^{\frac{d}{2}}}\frac{\left|y-\sqrt{1-t}x\right||x|}{t\sqrt{1-t}}dt\\
&=&
|x|\int_{0}^{1}\left(-log(\sqrt{1-t})\right)^{\frac{\beta}{2}}\frac{e^{-\epsilon
u(t)-(1-\epsilon)u(t)}}{t^{\frac{d+1}{2}}}\frac{\left|y-\sqrt{1-t}x\right|}{\sqrt{t}\sqrt{1-t}}dt\\
&\leq &
C_{\epsilon}|x|\int_{0}^{1}\left(-log(\sqrt{1-t})\right)^{\frac{\beta}{2}}\frac{e^{-(1-\epsilon)u(t)}}{t^{\frac{d+1}{2}}}\frac{dt}{\sqrt{1-t}}.
\end{eqnarray*}
Since 
$\displaystyle\frac{\left(-log(\sqrt{1-t})\right)^{\frac{\beta}{2}}}{\sqrt{1-t}}$
is continuous on $[0,1/2]$ and therefore bounded, we get
\begin{equation*}
\int_{0}^{1/2}\left(-log(\sqrt{1-t})\right)^{\frac{\beta}{2}}\frac{e^{-(1-\epsilon)u(t)}}{t^{\frac{d+1}{2}}}\frac{dt}{\sqrt{1-t}}
\leq C_{\beta}
\int_{0}^{1/2}\frac{e^{-(1-\epsilon)u(t)}}{t^{\frac{d+1}{2}}}dt,
\end{equation*}
and since $0<t<1/2$, $\frac{1}{\sqrt{t}}<\frac{1}{t}$ and then, by (\ref{desint}), we get  
\begin{eqnarray*}
\int_{0}^{1/2}\frac{e^{-(1-\epsilon)u(t)}}{t^{\frac{d+1}{2}}}dt&\leq
&e^{(1-\epsilon)|x|^{2}}\int_{0}^{1/2}\frac{e^{-(1-\epsilon)\frac{a}{t}}}{t^{\frac{d+1}{2}}}dt\\
&\leq
&e^{(1-\epsilon)|x|^{2}}\int_{0}^{1/2}\frac{e^{-(1-\epsilon)\frac{a}{t}}}{t^{\frac{d}{2}+1}}dt\\
&\leq
&e^{(1-\epsilon)|x|^{2}}C_{\epsilon}e^{-(1-\epsilon)a}= C_{\epsilon}e^{-(1-\epsilon)|y|^{2}}.
\end{eqnarray*}
Analogously,
\begin{eqnarray*}
\int_{1/2}^{1}\left(-log(\sqrt{1-t})\right)^{\frac{\beta}{2}}\frac{e^{-(1-\epsilon)u(t)}}{t^{\frac{d+1}{2}}}\frac{dt}{\sqrt{1-t}}
&\leq
&e^{(1-\epsilon)|x|^{2}}\int_{1/2}^{1}\left(-log(\sqrt{1-t})\right)^{\frac{\beta}{2}}\frac{e^{-(1-\epsilon)\frac{a}{t}}}{t^{\frac{d+1}{2}}}\frac{dt}{\sqrt{1-t}}\\
&\leq &C_{d}e^{(1-\epsilon)|x|^{2}}\int_{1/2}^{1}\left(-log(\sqrt{1-t})\right)^{\frac{\beta}{2}}\frac{e^{-(1-\epsilon)\frac{a}{t}}}{\sqrt{1-t}}dt,
\end{eqnarray*}
since $\displaystyle\frac{1}{t^{\frac{d+1}{2}}}$ is bounded on
$[1/2,1]$, also as $1/2<t<1$ we have $-\frac{a}{t}<-a$ and then by Lemma \ref{leminteg}
\begin{eqnarray*}
\int_{1/2}^{1}\left(-log(\sqrt{1-t})\right)^{\frac{\beta}{2}}\frac{e^{-(1-\epsilon)\frac{a}{t}}}{\sqrt{1-t}}dt&\leq
&e^{-(1-\epsilon)a}\int_{1/2}^{1}\left(-log(\sqrt{1-t})\right)^{\frac{\beta}{2}}\frac{dt}{\sqrt{1-t}}\\
&=&C_{\beta}e^{-(1-\epsilon)a}.
\end{eqnarray*}
Thus,
\begin{equation*}
\int_{1/2}^{1}\left(-log(\sqrt{1-t})\right)^{\frac{\beta}{2}}\frac{e^{-(1-\epsilon)u(t)}}{t^{\frac{d+1}{2}}}\frac{dt}{\sqrt{1-t}}
\leq C_{d,\beta}
e^{(1-\epsilon)|x|^{2}}e^{-(1-\epsilon)a}=Ce^{-(1-\epsilon)|y|^{2}}.
\end{equation*}
Therefore
$$I \leq C_\epsilon |x| e^{-(1-\epsilon)|y|^{2}}.$$
Now, using again inequality (\ref{expineq}), we get
\begin{eqnarray*}
II&=&
\int_{0}^{1}\left(-log(\sqrt{1-t})\right)^{\frac{\beta}{2}}\frac{e^{-u(t)}}{t^{\frac{d}{2}}}\frac{\left|y-\sqrt{1-t}x\right|^{2}}{2t^{2}}dt\\
&=&
\frac{1}{2}\int_{0}^{1}\left(-log(\sqrt{1-t})\right)^{\frac{\beta}{2}}\frac{e^{-\epsilon
u(t)-(1-\epsilon)u(t)}}{t^{\frac{d}{2}+1}}\frac{\left|y-\sqrt{1-t}x\right|^{2}}{t}dt\\
&\leq
&C_{\epsilon}\int_{0}^{1}\left(-log(\sqrt{1-t})\right)^{\frac{\beta}{2}}\frac{e^{-(1-\epsilon)u(t)}}{t^{\frac{d}{2}+1}}dt\\
&\leq
&C_{\epsilon}e^{(1-\epsilon)|x|^{2}}\int_{0}^{1}\left(-log(\sqrt{1-t})\right)^{\frac{\beta}{2}}\frac{e^{-(1-\epsilon)\frac{a}{t}}}{t^{\frac{d}{2}+1}}dt.\\
\end{eqnarray*}
Now,
\begin{equation*}
\int_{0}^{1/2}\left(-log(\sqrt{1-t})\right)^{\frac{\beta}{2}}\frac{e^{-(1-\epsilon)\frac{a}{t}}}{t^{\frac{d}{2}+1}}dt
\leq
C\int_{0}^{1/2}\frac{e^{-(1-\epsilon)\frac{a}{t}}}{t^{\frac{d}{2}+1}}dt
\leq
C\int_{0}^{1}\frac{e^{-(1-\epsilon)\frac{a}{t}}}{t^{\frac{d}{2}+1}}dt.
\end{equation*}
Therefore, by taking the change of variables $s=a(\frac{1}{t}-1)$, we get
\begin{eqnarray}\label{desint}
\nonumber\int_{0}^{1}\frac{e^{-(1-\epsilon)\frac{a}{t}}}{t^{\frac{d}{2}+1}}dt&=
&\int_{0}^{+\infty}e^{-(1-\epsilon)(s+a)}\left(\frac{s+a}{a}\right)^{\frac{d}{2}+1}\frac{a}{\left(s+a\right)^{2}}ds\\
\nonumber&=
&\frac{e^{-(1-\epsilon)a}}{a^{\frac{d}{2}}}\int_{0}^{+\infty}e^{-(1-\epsilon)s}\left(s+a\right)^{\frac{d}{2}-1}ds\\
\nonumber &\leq
&\frac{Ce^{-(1-\epsilon)a}}{a^{\frac{d}{2}}}\int_{0}^{+\infty}e^{-(1-\epsilon)s}\left(s^{\frac{d}{2}-1}+a^{\frac{d}{2}-1}\right)ds\\
\nonumber &\leq &\frac{C_{\epsilon}e^{-(1-\epsilon)a}}{a^{\frac{d}{2}}}\left(\Gamma\left(\frac{d}{2}\right)\;+\;a^{\frac{d}{2}-1}\right)\\
\nonumber &=&C_{\epsilon}e^{-(1-\epsilon)a}\left(\frac{\Gamma\left(\frac{d}{2}\right)}{a^{\frac{d}{2}}}\;+\;\frac{1}{a}\right)\\
&\leq &C_{\epsilon}e^{-(1-\epsilon)a},
\end{eqnarray}
since $a\geq\frac{d}{2}$. Analogously, by Lemma \ref{lemgamma}
\begin{eqnarray*}
\int_{1/2}^{1}\left(-log(\sqrt{1-t})\right)^{\frac{\beta}{2}}\frac{e^{-(1-\epsilon)}}{t^{\frac{d}{2}+1}}dt
&\leq &C\int_{ 1/2}^{1}\left(-log(\sqrt{1-t})\right)^{\frac{\beta}{2}}e^{-(1-\epsilon)\frac{a}{t}}dt\\
&\leq &C\int_{1/2}^{1}\left(-log(\sqrt{1-t})\right)^{\frac{\beta}{2}}e^{-(1-\epsilon)a}dt\\
&\leq&Ce^{-(1-\epsilon)a}\int_{0}^{1}\left(-log(\sqrt{1-t})\right)^{\frac{\beta}{2}}dt \\
&=& Ce^{-(1-\epsilon)a}.
\end{eqnarray*}
Then,
\begin{equation*}
II\leq
C_{\epsilon}e^{(1-\epsilon)|x|^{2}}e^{-(1-\epsilon)a}=C_{\epsilon}e^{(1-\epsilon)|x|^{2}}e^{-(1-\epsilon)\left(|y|^{2}+|x|^{2}\right)}=C_{\epsilon}e^{-(1-\epsilon)|y|^{2}}.
\end{equation*}
Finally, 
\begin{eqnarray*}
III&=&\int_{0}^{1}\left(-log(\sqrt{1-t})\right)^{\frac{\beta}{2}}\frac{e^{-u(t)}}{t^{\frac{d}{2}}}\frac{d}{2}\frac{dt}{t}\leq \frac{d}{2}\int_{0}^{1}\left(-log(\sqrt{1-t})\right)^{\frac{\beta}{2}}e^{-\frac{a}{t}+|x|^{2}}\frac{dt}{t^{\frac{d}{2}+1}}\\
&= &Ce^{|x|^{2}}\int_{0}^{1}\left(-log(\sqrt{1-t})\right)^{\frac{\beta}{2}}\frac{e^{-\frac{a}{t}}}{t^{\frac{d}{2}+1}}dt.\\
\end{eqnarray*}
Now,
\begin{equation*}
\int_{0}^{1/2}\left(-log(\sqrt{1-t})\right)^{\frac{\beta}{2}}\frac{e^{-\frac{a}{t}}}{t^{\frac{d}{2}+1}}dt\leq
C\int_{0}^{1/2}\frac{e^{-\frac{a}{t}}}{t^{\frac{d}{2}+1}}dt\leq
C\int_{0}^{1}\frac{e^{-\frac{a}{t}}}{t^{\frac{d}{2}+1}}dt
\end{equation*}
since $\left(-log(\sqrt{1-t})\right)^{\frac{\beta}{2}}$ is
bounded in $[0,1/2],$ by taking 
$$s=a(\frac{1}{t}-1), \;ds=-\frac{a}{t^{2}}dt,$$
we get
\begin{eqnarray*}
\int_{0}^{1}\frac{e^{-\frac{a}{t}}}{t^{\frac{d}{2}+1}}dt&=&
\int_{0}^{+\infty}e^{-(s+a)}\left(\frac{s+a}{a}\right)^{\frac{d}{2}+1}\left(\frac{a}{s+a}\right)^{2}\frac{ds}{a}\\
&=&\frac{e^{-|y|^{2}-|x|^{2}}}{a^{\frac{d}{2}}}\int_{0}^{+\infty}e^{-s}\left(s+a\right)^{\frac{d}{2}-1}ds.\\
\end{eqnarray*}
On the other hand,
\begin{equation*}
\left(s+a\right)^{\frac{d}{2}-1}=\frac{\left(s+a\right)^{\frac{d}{2}}}{s+a}\leq
C\frac{\left(s^{\frac{d}{2}}+a^{\frac{d}{2}}\right)}{s+a}\leq
C\left(\frac{s^{\frac{d}{2}}}{s}+\frac{a^{\frac{d}{2}}}{a}\right)
=C\left(s^{\frac{d}{2}-1}+a^{\frac{d}{2}-1}\right).
\end{equation*}
Therefore,
\begin{eqnarray*}
\int_{0}^{1}\frac{e^{-\frac{a}{t}}}{t^{\frac{d}{2}+1}}dt&\leq &
C\frac{e^{-|y|^{2}-|x|^{2}}}{a^{\frac{d}{2}}}\left(\int_{0}^{+\infty}e^{-s}s^{\frac{d}{2}-1}ds\;+
\;\int_{0}^{+\infty}e^{-s}a^{\frac{d}{2}-1}ds\right)\\
&=&C\frac{e^{-|y|^{2}-|x|^{2}}}{a^{\frac{d}{2}}}\left(\Gamma\left(\frac{d}{2}\right)\;+\;a^{\frac{d}{2}-1}\right)=Ce^{-|y|^{2}-|x|^{2}}\left(\frac{\Gamma\left(\frac{d}{2}\right)}{a^{\frac{d}{2}}}\;+\;\frac{1}{a}\right).
\end{eqnarray*}
Thus,
\begin{equation*}
\int_{0}^{1}\frac{e^{-\frac{a}{t}}}{t^{\frac{d}{2}+1}}dt\leq
Ce^{-|y|^{2}-|x|^{2}},\hspace{0.7cm}\mbox{as }a\geq\frac{d}{2}.
\end{equation*}
For $\frac{1}{2}<t<1,\hspace{0.3cm}-a>-\frac{a}{t}>-2a$. Hence, by Lemma \ref{lemgamma}
\begin{eqnarray*}
\int_{1/2}^{1}\left(-log(\sqrt{1-t})\right)^{\frac{\beta}{2}}\frac{e^{-\frac{a}{t}}}{t^{\frac{d}{2}+1}}dt&\leq
&C\int_{1/2}^{1}\left(-log(\sqrt{1-t})\right)^{\frac{\beta}{2}}e^{-a}dt\\
&\leq
&Ce^{-a}\int_{0}^{1}\left(-log(\sqrt{1-t})\right)^{\frac{\beta}{2}}dt\\
&=&Ce^{-|y|^{2}-|x|^{2}}.
\end{eqnarray*}
Then,
$$
III\leq Ce^{|x|^{2}}e^{-|y|^{2}-|x|^{2}}= Ce^{-|y|^{2}}\leq Ce^{-(1-\epsilon)|y|^{2}}.
$$
In other words,
\begin{equation*}
I\leq C_{\epsilon}|x|e^{-(1-\epsilon)|y|^{2}}
\end{equation*}
and
\begin{equation*}
II, \; III\leq Ce^{-|y|^{2}}\leq Ce^{-(1-\epsilon)|y|^{2}}.
\end{equation*}
Thus, for $b\leq 0$
\begin{equation*}
\left|N_{\beta/2}(x,y)\right|\leq
C_{\epsilon}(|x|+1)e^{-(1-\epsilon)|y|^{2}}.
\end{equation*}
Next, we take $0<\epsilon<1/p'_{-}$ and $\tilde{\epsilon}=1/p'_{-} - \epsilon = 1-\epsilon-1/p_{-}$. Then, $\tilde{\epsilon}>0$ and $1-\epsilon=\tilde{\epsilon}+1/p_{-}$\,.Therefore, for $f\in L^{p(\cdot)}(\mathbb{R}^{d},\gamma_{d})$ with $\|f\|_{p,\gamma_{d}}=1$, using H\"older's inequality
\begin{eqnarray*}
\int\limits_{\mathbb{R}^{d}}\left(\int\limits_{B_{h}^{c}(x)\cap
E^{c}_{x}}\left|N_{\beta/2}(x,y)\right||f(y)|dy\right)^{p(x)}\gamma_{d}(dx) &\lesssim
&\int\limits_{\mathbb{R}^{d}}\left(\int\limits_{\mathbb{R}^{d}}(|x|+1)e^{-(1-\epsilon)|y|^{2}}|f(y)|dy\right)^{p(x)}\gamma_{d}(dx)
\\ &=
&\int\limits_{\mathbb{R}^{d}}\left(\int\limits_{\mathbb{R}^{d}}e^{-(\tilde{\epsilon}+1/p_{-})|y|^{2}}|f(y)|dy\right)^{p(x)}(|x|+1)^{p(x)}\gamma_{d}(dx)
\\ &\leq
&\int\limits_{\mathbb{R}^{d}}\left(\int\limits_{\mathbb{R}^{d}}|f(y)|^{p_{-}}e^{-|y|^{2}}dy\right)^{p(x)/p_{-}}\left(\int\limits_{\mathbb{R}^{d}}e^{-\tilde{\epsilon}|y|^{2}p'_{-}}dy\right)^{p(x)/p'_{-}}\\
&&\hspace{5.5cm} \times (|x|+1)^{p_{+}}\gamma_{d}(dx)
\\ &\lesssim
&\int\limits_{\mathbb{R}^{d}}\left(\int\limits_{\mathbb{R}^{d}}|f(y)|^{p_{-}}e^{-|y|^{2}}dy\right)^{\frac{p(x)}{p_{-}}}(|x|+1)^{p_{+}}\gamma_{d}(dx),
\end{eqnarray*}
and, since $\rho_{p(\cdot),\gamma_{d}}(f)\leq 1$,
\begin{eqnarray*}
\int\limits_{\mathbb{R}^{d}}|f(y)|^{p_{-}}e^{-|y|^{2}}dy&\lesssim&\int\limits_{|f|\geq 1}|f(y)|^{p_{-}}\gamma_{d}(dy)+\int\limits_{|f|<1}|f(y)|^{p_{-}}\gamma_{d}(dy)\\
&\leq&\int\limits_{|f|\geq 1}|f(y)|^{p(y)}\gamma_{d}(dy)+1\leq\rho_{p(\cdot),\gamma_{d}}(f)+1\leq 2.
\end{eqnarray*}
Thus, 
$$\int\limits_{\mathbb{R}^{d}}\left(\int\limits_{B_{h}^{c}(x)\cap
E^{c}_{x}}\left|N_{\beta/2}(x,y)\right||f(y)|dy\right)^{p(x)}\gamma_{d}(dx) \lesssim
2^{\frac{p_{+}}{p_{-}}}\int\limits_{\mathbb{R}^{d}}(|x|+1)^{p_{+}}\gamma_{d}(dx)=C_{d,p}.$$

Hence, 
$$\|I_{\beta}(f\chi_{B_{h}^{c}(\cdot)\cap E^{c}_{(\cdot)}})\|_{p(\cdot),\gamma_{d}}\leq C_{d,p}.$$
\item \underline{Case $b>0$}:\\

In this case, $I$ is a very problematic term so we will discuss it at the end.\\

Again, by inequality (\ref{expineq})
\begin{eqnarray*}
II&=&\frac{1}{2}\int^{1}_{0}\left(-log(\sqrt{1-t})\right)^{\beta/2}\frac{e^{-u(t)}}{t^{d/2+1}}\frac{\left|y-\sqrt{1-t}x\right|^{2}}{t}\,dt\\
&=&\frac{1}{2}\int^{1}_{0}\left(-log(\sqrt{1-t})\right)^{\beta/2}\frac{e^{-(1-\epsilon)u(t)}}{t^{d/2+1}} e^{-\epsilon u(t)}\frac{\left|y-\sqrt{1-t}x\right|^{2}}{t}\,dt\\
&\leq &C_{\epsilon}\int^{1}_{0}\left(-log(\sqrt{1-t})\right)^{\beta/2}\frac{e^{-(1-\epsilon)u(t)}}{t^{d/2+1}}\,dt.
\end{eqnarray*}
By using inequality (4.44) of \cite{urbina2019}, see also \cite{pe}, we get
\begin{equation}\label{phib0}
\frac{e^{-u(t)}}{t^{d/2}}\leq 2^{d}\frac{e^{-u(t_{0})}}{t_{0}^{d/2}}.
\end{equation}
Thus,
\begin{eqnarray}\label{phib1}
\nonumber \frac{e^{-(1-\epsilon)u(t)}}{t^{d/2}}&=&\left(\frac{e^{-u(t)}}{t^{d/2}}\right)^{1-\epsilon} \frac{1}{t^{\epsilon
d/2}}\\
&\lesssim
&\left(\frac{e^{-u(t_{0})}}{t_{0}^{d/2}}\right)^{1-\epsilon} \frac{1}{t^{\epsilon
d/2}}=\frac{e^{-(1-\epsilon)u(t_{0})}}{t_{0}^{d(1-\epsilon)/2}}\frac{1}{t^{\epsilon
d/2}}.
\end{eqnarray}
Now, splitting the above integral  into two the integrals on $[0,1/2]$ and $[1/2,1]$; we have for the first integral using (\ref{phib1}),
\begin{eqnarray*}
\int^{1/2}_{0}\left(-log(\sqrt{1-t})\right)^{\beta/2}\frac{e^{-(1-\epsilon)u(t)}}{t^{d/2+1}}\,dt&\lesssim
&\frac{e^{-(1-\epsilon)u(t_{0})}}{t_{0}^{d(1-\epsilon)/2}}\int^{1/2}_{0}\frac{\left(-log(\sqrt{1-t})\right)^{\beta/2}}{t^{\epsilon
d/2+1}}\,dt.
\end{eqnarray*}
Set $r=\min\{\frac{\beta}{4},\frac{1}{2}\},\,$ then
$0<r<1$ and by taking $\epsilon>0$ such that $\frac{\epsilon d}{2}<
\frac{\beta}{2}-r\,\,$ we get
\begin{equation*}
\lim_{t\to 0^{+}}\frac{\left(-log(\sqrt{1-t})\right)^{\beta/2}}{t^{\epsilon
d/2+r}} = \lim_{t\to 0^{+}}\left[\frac{\left(-log(\sqrt{1-t})\right)}{t}\right]^{\beta/2}t^{\beta/2 -(\epsilon d/2+r)}= 0,
\end{equation*}
thus, $\displaystyle\frac{\left(-log(\sqrt{1-t})\right)^{\beta/2}}{t^{\epsilon
d/2+r}}$ is bounded on $(0,1/2]$,
and hence
\begin{equation*}
\int^{1/2}_{0}\frac{\left(-log(\sqrt{1-t})\right)^{\beta/2}}{t^{\epsilon
d/2+1}}\,dt=\int^{1/2}_{0}\frac{\left(-log(\sqrt{1-t})\right)^{\beta/2}}{t^{\epsilon
d/2+r}}\frac{dt}{t^{1-r}}\leq
C_{\epsilon,\beta}\int^{1/2}_{0}\frac{dt}{t^{1-r}}=C_{\epsilon,\beta}.
\end{equation*}
Then,
\begin{equation}\label{bceromedio}
\int^{1/2}_{0}\left(-log(\sqrt{1-t})\right)^{\beta/2}\frac{e^{-(1-\epsilon)u(t)}}{t^{d/2+1}}\,dt\;\leq\;C_{\epsilon,\beta}\frac{e^{-(1-\epsilon)u(t_{0})}}{t_{0}^{d(1-\epsilon)/2}}.
\end{equation}
For the integral on $[1/2,1]$ we have, using again (\ref{phib1}) and Lemma \ref{lemgamma}
\begin{eqnarray*}
\int^{1}_{1/2}\left(-log(\sqrt{1-t})\right)^{\beta/2}\frac{e^{-(1-\epsilon)u(t)}}{t^{d/2+1}}\,dt&=&
\int^{1}_{1/2}\left(-log(\sqrt{1-t})\right)^{\beta/2}\left(\frac{e^{-u(t)}}{t^{d/2}}\right)^{(1-\epsilon)}\,\frac{dt}{t^{\epsilon
d/2+1}}\\
&\leq
&C_{\epsilon}\frac{e^{-(1-\epsilon)u(t_{0})}}{t_{0}^{d(1-\epsilon)/2}}\int^{1}_{1/2}\left(-log(\sqrt{1-t})\right)^{\beta/2}\,dt\\
&\leq
&C_{\epsilon}\frac{e^{-(1-\epsilon)u(t_{0})}}{t_{0}^{d(1-\epsilon)/2}}\int^{1}_{0}\left(-log(\sqrt{1-t})\right)^{\beta/2}\,dt\\
&=&C_{\epsilon,\beta}\frac{e^{-(1-\epsilon)u(t_{0})}}{t_{0}^{d(1-\epsilon)/2}}.
\end{eqnarray*}
Therefore, since $t_{0}<1$ and
$\frac{d}{2}(1-\epsilon)<\frac{d}{2},$ we get
$$
II\leq C_{\epsilon,\beta}\frac{e^{-(1-\epsilon)u(t_{0})}}{t_{0}^{d(1-\epsilon)/2}}
\leq C_{\epsilon}\frac{e^{-(1-\epsilon)u(t_{0})}}{t_{0}^{d/2}}.
$$
Now, using (\ref{phib0}) and (\ref{gammaint}) we get
\begin{eqnarray*}
III=\frac{1}{4}\int^{1}_{0}\left(-log(\sqrt{1-t})\right)^{\beta/2}\frac{e^{-u(t)}}{t^{d/2}}\frac{dt}{t}&\lesssim&
\frac{e^{-u(t_{0})}}{t_{0}^{d/2}}\int^{1}_{0}\left(-log(\sqrt{1-t})\right)^{\beta/2}\frac{dt}{t}\\
&=&C_{\beta}\frac{e^{-u(t_{0})}}{t_{0}^{d/2}}\leq
C_{\beta}\frac{e^{-(1-\epsilon)u(t_{0})}}{t_{0}^{d/2}}.
\end{eqnarray*}
Now, to estimate $I$, we use again inequality (\ref{expineq})
\begin{eqnarray*}
I&=&\int^{1}_{0}\left(-log(\sqrt{1-t})\right)^{\beta/2}\frac{e^{-u(t)}}{t^{d/2}}\frac{\left|y-\sqrt{1-t}x\right|}{\sqrt{t}}\frac{|x|}{\sqrt{1-t}}\frac{dt}{\sqrt{t}}\\
&\leq
&C_{\epsilon}|x|\int^{1}_{0}\left(-log(\sqrt{1-t})\right)^{\beta/2}\frac{e^{-(1-\epsilon)u(t)}}{t^{d/2}\sqrt{t}}\frac{dt}{\sqrt{1-t}}.
\end{eqnarray*}
Again, splitting the above integral  into two the integrals on $[0,1/2]$ and $[1/2,1]$. For the second integral, using  (\ref{phib0}),  that $t \geq 1/2$ and Lemma \ref{leminteg}, we get
\begin{eqnarray*}
\int_{1/2}^{1}\left(-log(\sqrt{1-t})\right)^{\frac{\beta}{2}}\frac{e^{-(1-\epsilon)u(t)}}{t^{\frac{d}{2}}\sqrt{t}}\frac{dt}{\sqrt{1-t}}&\lesssim& e^{-(1-\epsilon)u(t_{0})}\int_{1/2}^{1}\left(-log(\sqrt{1-t})\right)^{\frac{\beta}{2}}\frac{dt}{\sqrt{1-t}}\\
&=&C_{\beta}e^{-(1-\epsilon)u(t_{0})},
\end{eqnarray*}

Next, for the integral on $[0,1/2]$ we need to consider two cases:\\
\begin{itemize}
\item Case $\beta>0$ and $d=1$: by Lemma \ref{lemgamma} and the fact that $\frac{-log(\sqrt{1-t})}{t}$ is bounded on $(0, 1/2],$
\begin{eqnarray*}
\int_{0}^{1/2}\left(-log(\sqrt{1-t})\right)^{\frac{\beta}{2}}\frac{e^{-(1-\epsilon)u(t)}}{t^{\frac{1}{2}}\sqrt{t}}\frac{dt}{\sqrt{1-t}}
&=&\int_{0}^{1/2}\left(-log(\sqrt{1-t})\right)^{\frac{\beta}{2}}\frac{e^{-(1-\epsilon)u(t)}}{t}\frac{dt}{\sqrt{1-t}}\\
&=&\int_{0}^{1/2}\left(-log(\sqrt{1-t})\right)^{\frac{\beta}{2}-1}\frac{\left(-log(\sqrt{1-t})\right)}{t}e^{-(1-\epsilon)u(t)}\frac{dt}{\sqrt{1-t}}\\
&\leq&C_{\beta}\int_{0}^{1/2}\left(-log(\sqrt{1-t})\right)^{\frac{\beta}{2}-1}e^{-(1-\epsilon)u(t)}\frac{dt}{\sqrt{1-t}}\\
&\leq&C_{\beta}e^{-(1-\epsilon)u(t_{0})}\int_{0}^{1/2}\left(-log(\sqrt{1-t})\right)^{\frac{\beta}{2}-1}dt\\
&=&C_{\beta}e^{-(1-\epsilon)u(t_{0})}.
\end{eqnarray*}
\item Case $\beta\geq 1$ and $d\geq 2$: by taking $\epsilon <
\frac{2}{d}$
\begin{eqnarray*}
\int_{0}^{1/2}\left(-log(\sqrt{1-t})\right)^{\frac{\beta}{2}}\frac{e^{-(1-\epsilon)u(t)}}{t^{\frac{d}{2}}\sqrt{t}}\frac{dt}{\sqrt{1-t}}
&\leq& C\int_{0}^{1/2}\left(-log(\sqrt{1-t})\right)^{\frac{(\beta-1)}{2}}\frac{e^{-(1-\epsilon)u(t)}}{t^{\frac{d}{2}}}\frac{dt}{\sqrt{1-t}}\\
&= &C\int_{0}^{1/2}\frac{\left(-log(\sqrt{1-t})\right)^{\frac{(\beta-1)}{2}}}{t^{\frac{(d-2)}{2}}}\frac{e^{-(1-\epsilon)u(t)}}{t\sqrt{1-t}}dt\\
&= &C\int_{0}^{1/2}\frac{\left(-log(\sqrt{1-t})\right)^{\frac{(\beta-1)}{2}}}{t^{\frac{d}{2}\frac{(d-2)}{d}}}e^{-(\frac{d-2}{d})u(t)}\frac{e^{(\epsilon-\frac{2}{d})u(t)}}{t\sqrt{1-t}}dt\\
&\leq &C\frac{e^{-(\frac{d-2}{d})u(t_{0})}}{t_{0}^{\frac{d}{2}\frac{(d-2)}{d}}}\int_{0}^{1/2}\left(-log(\sqrt{1-t})\right)^{\frac{(\beta-1)}{2}}\frac{e^{(\epsilon-\frac{2}{d})u(t)}}{t\sqrt{1-t}}dt.
\end{eqnarray*}
Since
$\left(-log(\sqrt{1-t})\right)^{\frac{(\beta-1)}{2}}$ is continuous
on $[0,1/2]$ for $\beta\geq 1$ and proceeding in analogous way as in Lemma 4.36 of \cite{urbina2019}, we get
\begin{eqnarray*}
\int_{0}^{1/2}\left(-log(\sqrt{1-t})\right)^{\frac{\beta}{2}}\frac{e^{-(1-\epsilon)u(t)}}{t^{\frac{d}{2}}\sqrt{t}}\frac{dt}{\sqrt{1-t}} &\leq& C_{\beta}\frac{e^{-(\frac{d-2}{d})u(t_{0})}}{t_{0}^{\frac{d}{2}\frac{(d-2)}{d}}}\int_{0}^{1}\frac{e^{(\epsilon-\frac{2}{d})u(t)}}{t\sqrt{1-t}}dt\\
&\leq &C_{\beta}\frac{e^{-(\frac{d-2}{d})u(t_{0})}}{t_{0}^{\frac{d}{2}-1}}e^{(\epsilon-\frac{2}{d})u(t_{0})}\\
&= &C_{\beta}\frac{e^{-(1-\epsilon)u(t_{0})}}{t_{0}^{\frac{d}{2}-1}}= C_{\beta}\frac{e^{-(1-\epsilon)u(t_{0})}}{t_{0}^{\frac{d}{2}}}t_{0}.
\end{eqnarray*}
Thus,
$$I\leq C_{\epsilon,\beta}\displaystyle|x|\left(1+\frac{1}{t_{0}^{\frac{d}{2}}}t_{0}\right)e^{-(1-\epsilon)u(t_{0})}.$$
\end{itemize}

Finally, since $\displaystyle\frac{1}{t_{0}^{d/2}}\lesssim |x+y|^{d}$, $t_{0}<1$, and $|x|\leq |x+y|$ as $b>0$, we have
\begin{itemize}
\item For $|x|<1$,
\begin{eqnarray*}
 I&\leq& C_{\epsilon,\beta}\displaystyle|x|\left(1+\frac{1}{t_{0}^{\frac{d}{2}}}t_{0}\right)e^{-(1-\epsilon)u(t_{0})}\leq C_{\epsilon,\beta}\frac{1}{t_{0}^{\frac{d}{2}}}e^{-(1-\epsilon)u(t_{0})}\\
 &\leq& C_{\epsilon,\beta}|x+y|^{d}e^{-(1-\epsilon)u(t_{0})}.
 \end{eqnarray*}
\item For $|x|\geq 1$, since $b>0$, $|x|\leq |x+y|$
\begin{eqnarray*}
I&\leq & C_{\epsilon,\beta}\displaystyle|x|\left(1+\frac{1}{t_{0}^{\frac{d}{2}}}t_{0}\right)e^{-(1-\epsilon)u(t_{0})}\\
&=& C_{\epsilon,\beta}|x|e^{-(1-\epsilon)u(t_{0})}\,+\, C_{\epsilon,\beta}|x|t_{0}\frac{e^{-(1-\epsilon)u(t_{0})}}{t_{0}^{\frac{d}{2}}}\\
&\leq &
 C_{\epsilon,\beta}|x+y|^{d}e^{-(1-\epsilon)u(t_{0})}\,+\, C_{\epsilon,\beta}t_{0}|x||x+y|^{d}e^{-(1-\epsilon)u(t_{0})}.
\end{eqnarray*}
\end{itemize}
Given that  $t_{0}\leq C\frac{|x-y|}{|x+y|}$ and the fact that  $|x|\leq |x+y| $ we get, for $|x-y|<1$,
$$
|x|t_{0}\leq C\frac{|x||x-y|}{|x+y|}\leq C.
$$
 Thus, for $|x-y|<1$,
$$
 I\leq
C_{\epsilon,\beta}|x+y|^{d}e^{-(1-\epsilon)u(t_{0})};
$$
and for $|x-y|\geq 1$,
$$
I\leq
C_{\epsilon,\beta}|x+y|^{d+1}e^{-(1-\epsilon)u(t_{0})}.
$$

Hence, we conclude that\\ 
\begin{itemize}
\item $\left|N_{\beta/2}(x,y)\right|\leq C_{\epsilon,\beta}|x+y|^{d}e^{-(1-\epsilon)u(t_{0})},$ \; for either $|x|\leq 1,$ or for $|x|\geq 1$ with $|x-y|<1$.\\

\item $\left|N_{\beta/2}(x,y)\right|\leq C_{\epsilon,\beta}|x+y|^{d+1}e^{-(1-\epsilon)u(t_{0})},$ \; for \; $|x|\geq 1$ with $|x-y|\geq 1$.\\
\end{itemize}

Now, as  $b>0$, for $f\in L^{p(\cdot)}(\mathbb{R}^{d},\gamma_{d})$ with $\|f\|_{p(\cdot),\gamma_{d}}=1$, we have that
\begin{eqnarray*}
  &&  \int_{\mathbb{R}^{d}}\left(\int\limits_{B_{h}^{c}(x)\cap
E_{x}}\left|N_{\beta/2}(x,y)\right||f(y)|dy\right)^{p(x)}\gamma_{d}(dx)= \int_{|x|<1}\left(\int\limits_{B_{h}^{c}(x)\cap
E_{x}}\left|N_{\beta/2}(x,y)\right||f(y)|dy\right)^{p(x)}\gamma_{d}(dx)\\
 &&  \hspace{.5cm}+ \int_{|x|\geq 1}\left(\int\limits_{B_{h}^{c}(x)\cap
E_{x},|x-y|<1}\left|N_{\beta/2}(x,y)\right||f(y)|dy+\int\limits_{B_{h}^{c}(x)\cap
E_{x},|x-y|\geq 1}\left|N_{\beta/2}(x,y)\right||f(y)|dy\right)^{p(x)}\gamma_{d}(dx)
\end{eqnarray*}
\begin{eqnarray*}
 &\leq& \int_{|x|<1}\left(\int\limits_{B_{h}^{c}(x)\cap
E_{x}}\left|N_{\beta/2}(x,y)\right||f(y)|dy\right)^{p(x)}\gamma_{d}(dx)\\
 &+&  C\int_{|x|\geq 1}\left(\left(\int\limits_{B_{h}^{c}(x)\cap
E_{x},|x-y|<1}\left|N_{\beta/2}(x,y)\right||f(y)|dy\right)^{p(x)}+\left(\int\limits_{B_{h}^{c}(x)\cap
E_{x},|x-y|\geq 1}\left|N_{\beta/2}(x,y)\right||f(y)|dy\right)^{p(x)}\right)\gamma_{d}(dx)\\
 &\leq&C_{\epsilon,\beta}\int_{\Bbb R^{d}}\left(\int\limits_{B_{h}^{c}(x)\cap
E_{x}}|x+y|^{d}e^{-(1-\epsilon)u(t_{0})}|f(y)|dy\right)^{p(x)}\gamma_{d}(dx)\\
  &&  \hspace{0.75cm}+ C_{\epsilon,\beta}\int_{\Bbb R^{d}}\left(\int\limits_{B_{h}^{c}(x)\cap
E_{x},|x-y|\geq 1}|x+y|^{d+1}e^{-(1-\epsilon)u(t_{0})}|f(y)|dy\right)^{p(x)}\gamma_{d}(dx)\\
&=&C_{\epsilon,\beta}\int_{\Bbb R^{d}}\left(\int\limits_{B_{h}^{c}(x)\cap
E_{x}}|x+y|^{d}e^{-(1-\epsilon)u(t_{0})}e^{|y|^{2}/p(y)-|x|^{2}/p(x)}|f(y)|e^{-|y|^{2}/p(y)}dy\right)^{p(x)}dx\\
  &&  \hspace{0.75cm}+ C_{\epsilon,\beta}\int_{\Bbb R^{d}}\left(\int\limits_{B_{h}^{c}(x)\cap
E_{x},|x-y|\geq 1}|x+y|^{d+1}e^{-(1-\epsilon)u(t_{0})}e^{|y|^{2}/p(y)-|x|^{2}/p(x)}|f(y)|e^{-|y|^{2}/p(y)}dy\right)^{p(x)}dx.  
\end{eqnarray*}
Since $p(\cdot)\in\mathcal{P}_{\gamma_{d}}^{\infty}(\mathbb{R}^{d})$, we obtain that $e^{|y|^{2}/p(y)-|x|^{2}/p(x)}\approx e^{(|y|^{2}-|x|^{2})/p_{\infty}},$ and by the Cauchy-Schwartz inequality we have,
 $\left| |y|^{2}-|x|^{2}\right|\leq|x+y| |x-y|,$
  for all $x,y\in\mathbb{R}^{d}$. Therefore,
\begin{eqnarray*}
  &&\int_{B^{c}_h(x)\cap E_{x}}|x+y|^{d}e^{-(1-\epsilon)u(t_{0})}e^{|y|^{2}/p(y)-|x|^{2}/p(x)}|f(y)|e^{-|y|^{2}/p(y)}dy \hspace{3.5cm}\\
    && \hspace{6.5cm} \lesssim \int_{B^{c}_h(x)\cap E_{x}}P(x,y)|f(y)|e^{-|y|^{2}/p(y)}dy,
\end{eqnarray*}
and \begin{eqnarray*}
  &&\int_{B^{c}_h(x)\cap E_{x},|x-y|\geq 1}|x+y|^{d+1}e^{-(1-\epsilon)u(t_{0})}e^{|y|^{2}/p(y)-|x|^{2}/p(x)}|f(y)|e^{-|y|^{2}/p(y)}dy \hspace{3.5cm}\\
    && \hspace{5.5cm} \lesssim \int_{B^{c}_h(x)\cap E_{x}}Q(x,y)|f(y)|e^{-|y|^{2}/p(y)}dy,
\end{eqnarray*}
where
$$
  P(x,y)=|x+y|^{d}e^{-\alpha_{\infty}|x+y||x-y|}, \; Q(x,y)=|x+y|^{d+1}e^{-\alpha_{\infty}|x+y|}$$
and  
$$ \alpha_{\infty}=\left(\frac{1-\epsilon}{2}-\left|\frac{1}{p_{\infty}}-\frac{1-\epsilon}{2}\right|\right).$$
It is easy to see that $\alpha_{\infty}>0$  if $\epsilon<1/p'_{\infty}$. 

Therefore,  in order to make sense of all the estimates above we need to take 
$$0<\epsilon<\displaystyle\min\left\lbrace \frac{1}{d},\frac{1}{p'_{\infty}},\frac{2}{d}\left(\frac{\beta}{2}-r\right)\right\rbrace.$$
Observe that $P(x,y)$ is the same kernel considered in the proof of Theorem 3.5, page 416 of \cite{DalSco}, so we can conclude that
$$\displaystyle\int_{\Bbb R^{d}}\left(\int_{B^{c}_h(x)\cap E_{x}}P(x,y)|f(y)|e^{-|y|^{2}/p(y)}dy\right)^{p(x)}dx\leq C.$$
On the other hand, it can be proved that $Q(x,y)$ is integrable on each variable and the value of each integral is independent of $x$ and $y$.
Now, we use an analogous argument as in \cite{DalSco} for $Q(x,y)$. Taking, 
$$J=\displaystyle\int_{\Bbb R^{d}}Q(x,y)|f(y)|e^{-|y|^{2}/p(y)}dy,$$
and using H\"older's inequality, we obtain
 $$ J\lesssim\|Q(x,\cdot)\|_{p'(\cdot)}\|fe^{-|\cdot|^{2}/p(\cdot)}\|_{p(\cdot)}\leq \|Q(x,\cdot)\|_{p'(\cdot)}.$$
and,
 \begin{eqnarray*}
 \int_{\Bbb R^{d}}Q(x,y)^{p'(y)}dy&=&\int_{\Bbb R^{d}}|x+y|^{(d+1)p'(y)}e^{-\alpha_{\infty}|x+y|p'(y)}dy\\
 &\leq&\int_{|x+y|<1}|x+y|^{d+1}e^{-\alpha_{\infty}|x+y|}dy +\int_{|x+y|\geq 1}|x+y|^{(d+1)p'_{+}}e^{-\alpha_{\infty}|x+y|}dy\\
 &\leq&\int_{\Bbb R^{d}}\left(|z|^{d+1}+|z|^{(d+1)p'_{+}}\right)e^{-\alpha_{\infty}|z|}dz\\
&=&C_{p,d}.
  \end{eqnarray*}
  Thus, $J\lesssim\|Q(x,\cdot)\|_{p^{'}(\cdot)}\leq C_{p,d}$, and therefore
   $$\frac{1}{C_{p,d}}\int_{\Bbb R^{d}}Q(x,y)|f(y)|e^{-|y|^{2}/p(y)}dy\leq 1.$$
We set $g(y)=|f(y)|e^{-|y|^{2}/p(y)}=g_{1}(y)+g_{2}(y)$, where $g_{1}=g\chi_{\{g\geq 1\}}$ and $g_{2}=g\chi_{\{g<1\}}$, then
\begin{eqnarray*}
  \int_{\mathbb{R}^{d}}J^{p(x)}dx&=& \int_{\mathbb{R}^{d}}\left(\int_{\mathbb{R}^{d}}Q(x,y)|f(y)|e^{-|y|^{2}/p(y)}dy\right)^{p(x)}dx \\
      &\lesssim& \int_{\mathbb{R}^{d}}\left(\frac{1}{C_{p,d}}\int_{\mathbb{R}^{d}}Q(x,y)g(y)dy\right)^{p(x)}dx\\
   &\lesssim& \int_{\mathbb{R}^{d}}\left(\frac{1}{C_{p,d}}\int_{\mathbb{R}^{d}}Q(x,y)g_{1}(y)dy\right)^{p(x)}dx+\int_{\mathbb{R}^{d}}\left(\frac{1}{C_{p,d}}\int_{\mathbb{R}^{d}}Q(x,y)g_{2}(y)dy\right)^{p(x)}dx.
 \end{eqnarray*}
By H\"{o}lder's inequality and Fubini's theorem 
\begin{eqnarray*}
 \int_{\mathbb{R}^{d}}\left(\frac{1}{C_{p,d}}\int_{\mathbb{R}^{d}}Q(x,y)g_{1}(y)dy\right)^{p(x)}dx&\lesssim&\int_{\mathbb{R}^{d}}\left(\int_{\mathbb{R}^{d}}Q(x,y)g_{1}(y)dy\right)^{p_{-}}dx\\
 &=&\int_{\mathbb{R}^{d}}\left(\int_{\mathbb{R}^{d}}Q^{\frac{1}{p'_{-}}}(x,y)Q^{\frac{1}{p_{-}}}(x,y)g_{1}(y)dy\right)^{p_{-}}dx\\
 &\leq&\int_{\mathbb{R}^{d}}\left(\int_{\mathbb{R}^{d}}Q(x,y)dy\right)^{p_{-}/p'_{-}}\left(\int_{\mathbb{R}^{d}}Q(x,y)g^{p_{-}}_{1}(y)dy\right)dx\\
 &=&C_{p}\int_{\mathbb{R}^{d}}\left(\int_{\mathbb{R}^{d}}Q(x,y)g^{p_{-}}_{1}(y)dy\right)dx\\
  &=&C_{p}\int_{\mathbb{R}^{d}}g^{p_{-}}_{1}(y)\left(\int_{\mathbb{R}^{d}}Q(x,y)dx\right)dy=C_{p}\int_{\mathbb{R}^{d}}g^{p_{-}}_{1}(y)dy.
 \end{eqnarray*}
 Therefore, 
 \begin{eqnarray*}
 \int_{\mathbb{R}^{d}}\left(\frac{1}{C_{p,d}}\int_{\mathbb{R}^{d}}Q(x,y)g_{1}(y)dy\right)^{p(x)}dx&\lesssim&\int_{\mathbb{R}^{d}}g^{p(y)}_{1}(y)dy\leq\int_{\mathbb{R}^{d}}|f(y)|^{p(y)}e^{-|y|^{2}}dy\\
 &\lesssim& \rho_{p(\cdot),\gamma_{d}}(f)\leq 1.
 \end{eqnarray*}
On the other hand, applying the inequality (\ref{3.26.1}) in Lemma \ref{lema3.26CU}, since\\
$G(x):=\displaystyle\frac{1}{C_{p,d}}\int_{\mathbb{R}^{d}}Q(x,y)g_{2}(y)dy\leq 1$, we obtain that
\begin{eqnarray*}
 \int_{\mathbb{R}^{d}}\left(\int_{\mathbb{R}^{d}}\frac{1}{C_{p,d}}Q(x,y)g_{2}(y)dy\right)^{p(x)}dx&=&\int_{\mathbb{R}^{d}}(G(x))^{p(x)}dx \\
  &\leq& \int_{\mathbb{R}^{d}}(G(x))^{p_{\infty}}dx + \int_{\mathbb{R}^{d}}\frac{dx}{(e+|x|)^{dp_{-}}}\\
  &\lesssim&\int_{\mathbb{R}^{d}}\left(\int_{\mathbb{R}^{d}}Q(x,y)g_{2}(y)dy\right)^{p_{\infty}}+C_{d,p}.
\end{eqnarray*}

Finally, to estimate the integral 
$$\int_{\mathbb{R}^{d}}\left(\int_{\mathbb{R}^{d}}Q(x,y)g_{2}(y)dy\right)^{p_{\infty}}dx,$$
we proceed in an analogous way, applying H\"{o}lder's inequality to the exponent $p_{\infty}$, Fubbini's theorem and inequality (\ref{3.26.2}) in Lemma \ref{lema3.26CU}, we get
\begin{eqnarray*}
  \int_{\mathbb{R}^{d}}\left(\int_{\mathbb{R}^{d}}Q(x,y)g_{2}(y)dy\right)^{p_{\infty}}dx &\lesssim&\int_{\mathbb{R}^{d}}g_{2}^{p_{\infty}}(y)dy\\
&\leq&\int_{\mathbb{R}^{d}}g_{2}^{p(y)}(y)dy+\int_{\mathbb{R}^{d}}\frac{dy}{(e+|y|)^{dp_{-}}}\\
&\leq&\int_{\mathbb{R}^{d}}|f(y)|^{p(y)}e^{-|y|^{2}}dy+C_{d,p}\lesssim 1+C_{d,p}.
\end{eqnarray*}
Thus,
\begin{eqnarray*}
 \int_{\mathbb{R}^{d}}\left(\int\limits_{B_{h}^{c}(x)\cap
E_{x}}\left|N_{\beta/2}(x,y)\right||f(y)|dy\right)^{p(x)}\gamma_{d}(dx)&\leq& C_{d,p}.
\end{eqnarray*}
With this, we obtain that $\|I_{\beta}(f\chi_{B_{h}^{c}(\cdot)\cap E_{(\cdot)}})\|_{p(\cdot),\gamma_{d}}\leq C_{d,p}$.\\
We conclude that
$$
\left\|I_{\beta,G}(f)\right\|_{p(\cdot), \gamma_{d}}=\left\|I_{\beta}\left(f \chi_{B^{c}_h(\cdot)}\right)\right\|_{p(\cdot), \gamma_{d}} \leq C,
$$\\
and by homogeneity of the norm,\\

$
\left\|I_{\beta,G}(f)\right\|_{p(\cdot), \gamma_{d}} \leq C\|f\|_{p(\cdot),\gamma_{d}}$, for all function $f\in L^{p(\cdot)}(\mathbb{R}^{d},\gamma_{d})$.\\

Therefore,  the global part $I_{\beta,G}$ is bounded in $L^{p(\cdot)}(\gamma_d)$ and the proof is complete. \\
\end{itemize}

\end{itemize}

\end{proof}


\begin{thebibliography}{99}
\bibitem{AdaHarjuHas}
Adamowicz, T.  Harjulehto, P. and  H\"{a}st\"{o}, P. \textit{Maximal Operator in Variable Exponent Lebesgue Spaces on Unbounded Quasimetric Measure Space.} Math. Scand. 116 (2015), no.1, pp 5-22.
%


\bibitem{Berez}
Berezhno\v{i} E.I. \textit{Two-weighted estimations for the Hardy-Littlewood maximal function in ideal
Banach spaces}. Proc Amer Math Soc. 1999;127(1):79–87; Disponible en: http://dx.doi.org/
10.1090/S0002-9939-99-04998-9

\bibitem{dcruz1}
Cruz-Uribe, D. \& Fiorenza, A., Neugebauer, C. J. \textit{The maximal function on variable $L^p$ spaces.} Ann Acad Sci Fenn Math. 2003 28(1) 223–238.

\bibitem{dcruz}
Cruz-Uribe, D. \& Fiorenza, A.\textit{Variable Lebesgue Spaces 
Foundations and Harmonic Analysis}, Applied and Numerical Harmonic Analysis Birkh\"auser-Springer, Basel, (2013)

\bibitem{DalSco}
 Dalmasso, E. \& Scotto, R. (2017) \textit{Riesz transforms on variable
Lebesgue spaces with Gaussian measure}, Integral Transforms and Special Functions, 28:5,
403-420, DOI: 10.1080/10652469.2017.1296835

\bibitem{LibroDenHarjHas}
 Diening, L.,  Harjulehto, P., \"{a}st\"{o},P. H,  and  R\.{u}\v{z}i\v{c}ka, M. \textit{Lebesgue and Sobolev spaces with variable exponents}. Lecture Notes in Mathematics, 2017. Springer, Heidelberg, 2011.

\bibitem{duo}
Duoandikoetxea, J. {\em Fourier Analysis.} Graduated Studies in Mathematics, Volume 29, AMS R.I (2001).

\bibitem{grafak} 
 Grafakos, L.  \textit{  Classical Fourier Analysis} GTM 249-50. 2nd. edition. Springer-Verlag (2008).
 \bibitem{MorPiUrb} Moreno, J., Pineda, E., \& Urbina, W.  \textit{Boundedness of the maximal function of the Ornstein-Uhlenbeck semigroup
on variable Lebesgue spaces with respect to the Gaussian measure and consequences.} Revista Colombiana de Matem\'aticas. Vol. 55 N\'um. 1, 21-41 (2021). 

\bibitem{pe} 
S. P\'erez \textit{ The local part and the strong type for operators related to the Gauss measure.} J. Geom. Anal. 11 (2001), no. 3, 491--507. MR1857854 (2002h:42027)

\bibitem{TesSon}
  P\'{e}rez, S. \textit{Estimaciones puntuales y en normas para operadores relacionados con el semigrupo de Ornstein-Uhlenbeck}, Memorias para optar al t\'{i}tulo de Doctora, Departamento de Matem\'{a}ticas, Universidad Aut\'{o}noma de Madrid.

\bibitem{st1}
Stein E. \textit{Singular Integrals and Differentiability Properties
of Functions}, Princeton Univ. Press. Princeton, New Jersey. (1970).

\bibitem{ur1}
Urbina, W. \textit{ Singular Integrals  with respect to the Gaussian measure.} Scuola Normale Superiore di Pisa. Classe di Science. Serie IV Vol XVII, 4 (1990) 531--567. MR1093708 (92d:42010)

\bibitem{urbina2019}
 Urbina W. \textit{Gaussian Harmonic Analysis}, Springer Monographs in Math. Springer Verlag, Switzerland AG (2019). 
 
%
\end{thebibliography}
\end{document}